%% file: main.tex
\documentclass[11pt]{article}

\usepackage[utf8]{inputenc} 
\usepackage[T1]{fontenc}    
\usepackage{url}            
\usepackage{booktabs}       
\usepackage{nicefrac}       
\usepackage{microtype}      
\usepackage{multirow}
\usepackage{array}
\usepackage[noblocks]{authblk}
\usepackage{amssymb,arydshln}
\usepackage[nodayofweek]{datetime}
\usepackage{float}
\usepackage{bm}
\usepackage{multicol} 
\usepackage{tocloft}  
\usepackage{arydshln}
\usepackage{cellspace} 

\usepackage[shortlabels]{enumitem}
\setlist[itemize]{itemindent=0ex,parsep=3pt,itemsep=0pt,leftmargin=\parindent,topsep=5pt,labelwidth=0.8em,labelsep=0.7em}
\setlist[enumerate]{label={\arabic*)},itemindent=0ex,parsep=3pt,itemsep=0pt,leftmargin=\parindent,topsep=5pt,labelwidth=0.9em,labelsep=0.6em}

\usepackage[font=small,labelfont=bf]{caption}
\usepackage{amsmath,amsfonts,nccmath,mathtools,mathrsfs}
\usepackage{tcolorbox}
\usepackage{xcolor}

\usepackage[margin=1in]{geometry}

\usepackage{pgfplots}
\usepgfplotslibrary{fillbetween}
\usetikzlibrary{arrows.meta}
\tikzset{>=latex}

\DeclarePairedDelimiterX\setv[2]{\{}{\}}{#1 \;\delimsize\vert\; #2}

\pgfmathdeclarefunction{cubic}{1}{%
  \pgfmathparse{-2*(#1+2)*(#1+2)*(#1-2)+40}%
}

\pgfmathdeclarefunction{bicuadratic}{1}{%
  \pgfmathparse{(#1-1)*(#1-1)*(#1-1)*(#1-1)+10}%
}

\pgfmathdeclarefunction{cuadratic}{1}{%
    \pgfmathparse{(-(2*#1)^2)+70}%
}

\newcommand{\ct}{\texttt{ct}}%
\newcommand{\dt}{\texttt{dt}}%

\usepackage[numbers,merge,sort&compress]{natbib}
\usepackage{amsthm}
\usepackage{graphicx,color}
\usepackage{subcaption,setspace}

\usepackage[colorlinks = true,
            linkcolor = blue,
            urlcolor  = blue,
            citecolor = blue,
            anchorcolor = blue]{hyperref}

\newtheorem{theorem}{Theorem}
\newtheorem{proposition}{Proposition}
\newtheorem{lemma}{Lemma}
\newtheorem{corollary}{Corollary}

\newtheorem{assumption}{Assumption}

\theoremstyle{definition}
\newtheorem{definition}{Definition}
\newtheorem{example}{Example}

\newtheorem{remark}{Remark}

\usepackage{appendix}
\usepackage[nameinlink]{cleveref}

\crefname{equation}{}{}
\crefname{theorem}{Theorem}{Theorems}
\crefname{corollary}{Corollary}{Corollaries}
\crefname{example}{Example}{Examples}
\crefname{assumption}{Assumption}{Assumptions}
\crefname{lemma}{Lemma}{Lemmas}
\crefname{proposition}{Proposition}{Propositions}
\crefname{figure}{Figure}{Figures}
\crefname{table}{Table}{Tables}
\crefname{fact}{Fact}{Facts}
\crefname{conjecture}{Conjecture}{Conjectures}
\crefname{section}{Section}{Sections}
\crefname{appendix}{Appendix}{Appendices}
\Crefname{equation}{}{}
\Crefname{theorem}{Theorem}{Theorems}
\Crefname{corollary}{Corollary}{Corollaries}
\Crefname{example}{Example}{Examples}
\Crefname{lemma}{Lemma}{Lemma}
\Crefname{proposition}{Proposition}{Proposition}
\Crefname{figure}{Figure}{Figures}
\Crefname{table}{Table}{Tables}
\Crefname{section}{Section}{Sections}
\Crefname{appendix}{Appendix}{Appendices}

\input{commands}

\graphicspath{{figure/}}

\usepackage{tikz}
\usepackage{pgfplots}
\pgfplotsset{compat=newest}

\makeatletter
\newcommand{\removelatexerror}{\let\@latex@error\@gobble}
\makeatother

\title{\bf \Large
Gradient Dominance in the Linear Quadratic Regulator:
A Unified Analysis for Continuous-Time and Discrete-Time Systems
\thanks{This work is supported by NSF CMMI 2320697 and NSF CAREER 2340713. Emails: \texttt{y1watanabe@ucsd.edu; zhengy@ucsd.edu}.}
}

\author[1]{Yuto Watanabe} 
\author[1]{Yang Zheng}
\affil[1]{\small Department of Electrical and Computer Engineering, University of California San Diego \vspace{-2mm}}

\date{ \small \today \vspace{-3mm}} 

\begin{document}

\maketitle
\vspace{-5mm}

\begin{abstract}
Despite its nonconvexity, policy optimization for the Linear Quadratic Regulator (LQR) admits a favorable structural property known as \textit{gradient dominance}, which facilitates linear convergence of policy gradient methods to the globally optimal gain. While gradient dominance has been extensively studied, continuous-time and discrete-time LQRs have largely been analyzed separately, relying on slightly different assumptions, proof strategies, and resulting guarantees.
In this paper, we present a unified gradient dominance property for both continuous-time and discrete-time LQRs under mild stabilizability and detectability assumptions. 
Our analysis is based on 
a convex reformulation derived from 
a common Lyapunov inequality representation and a unified change-of-variables procedure. 
This convex-lifting perspective yields a single proof framework applicable to both time models. The unified treatment clarifies how differences between continuous-time and discrete-time dynamics influence theoretical guarantees and reveals a deeper structural symmetry between the two formulations. Numerical examples illustrate and support the theoretical findings. 
\noindent 
\end{abstract} 

\input{sections/section-1}

\input{sections/section-2}

\input{sections/section-3}

\input{sections/section-4}

\section{Conclusion}\label{section:conclusion}

This paper provides a unified treatment of gradient dominance for continuous-time and discrete-time LQRs. Our analysis leverages the hidden convexity of LQR via an inequality-constrained convex reformulation that places the two time models in a common framework. Within this perspective, we show that gradient dominance follows from two elementary bounds on the optimality gap, offering a transparent explanation of why both formulations exhibit benign optimization landscapes despite nonconvex policy parameterizations. This unified view also helps isolate which aspects of the guarantees are intrinsic to LQR and which depend on the underlying time model. Future work will explore extensions of these ideas to policy optimization for robust and output-feedback control, as well as broader reinforcement-learning settings.

\addcontentsline{toc}{section}{References}
{
\bibliographystyle{ieeetr}
\bibliography{ref.bib}
}


\numberwithin{equation}{section}
\numberwithin{example}{section}
\numberwithin{remark}{section}
\numberwithin{assumption}{section}
\numberwithin{lemma}{section}
\numberwithin{proposition}{section}


\newpage

\appendix 
\vspace{10mm}
\noindent\textbf{\Large Appendix}

\crefalias{section}{appendix}
\crefalias{subsection}{appendix}
\crefalias{subsubsection}{appendix}

In the appendix, we provide technical backgrounds, proofs, and further discussions.
The appendix is organized as follows.
\begin{itemize}
    \item \cref{appendix:backgrounds_sec2-3} provides some technical backgrounds and proofs for 
\cref{section:problem_formulation,section:gradient_dominance}.
    \item \cref{appendix:proof}
    provides 
    technical proofs for 
\cref{section:proof_PL,section:proof_partialmin_chain-rule}.
\end{itemize}

\input{sections/appendix-technical-details}

\input{sections/appendix-proofs}

\end{document}

%% file: commands.tex
\newcommand{\tr}{{{\mathsf T}}}
\newcommand{\Tr}{{{\mathbf{Tr}}}}
\newcommand{\her}{{{\mathsf H}}}

\newcommand{\yang}[1]{\textcolor{red}{[YZ: #1]}}

%% file: sections/section-1.tex
\section{Introduction}
\label{sec:introduction}

The Linear Quadratic Regulator (LQR) is one of the most foundational problems in optimal control and a cornerstone of modern control theory and engineering practice \cite{kalman1960contributions,bellman1966dynamic}. A vast body of theoretical developments and practical applications builds upon the LQR framework, 
e.g., robust control \cite{zhou1996robust}, and model predictive control \cite{rawlings2017model}. 
In the LQR problem, the objective is to design an optimal controller for a linear dynamical system by minimizing the integral (in continuous time) or the sum (in discrete time) of a quadratic cost function along system trajectories.
A remarkable and elegant fact is that for the infinite-horizon formulation,  the optimal control law is linear and static, taking the form 
$u=Kx$, 
regardless of whether the underlying system evolves in continuous time or discrete time.

Policy optimization for LQR leverages the simple form of the optimal policy and yields striking 
theoretical guarantees \cite{fazel2018global,mohammadi2019global,bu2019lqr,bu2020policy,fatkhullin2021optimizing,watanabe2025revisiting}.
Together with the broader success of policy optimization in reinforcement learning, it has emerged as a powerful paradigm for model-free control design \cite{hu2023toward,talebi2024policy}. Unlike Riccati-based approaches \cite{lancaster1995algebraic} or convex reformulations \cite{boyd1994linear}, policy optimization treats controller design as a nonconvex optimization problem over the feedback gain $K$ and seeks to solve it via policy gradient methods: 
\begin{equation}\label{eq:LQR_PO_intro}
    \min_{K} \quad J(K)\quad
    \text{subject to}\quad K\in \mathcal{K}, 
\end{equation}
where $J(K)$ is the LQR cost associated with $K$ and 
$\mathcal{K}$ is the set of stabilizing gains.
Despite the inherent nonconvexity, a number of recent works \cite{fazel2018global,bu2019lqr,bu2020policy,mohammadi2019global,fatkhullin2021optimizing,watanabe2025revisiting,de2025remarks} have demonstrated that policy optimization for LQR enjoys \textit{linear convergence} to
\textit{the global minimizer} under mild assumptions. 
One key enabler of this guarantee is a structural property known as \textit{gradient dominance}, also referred to as the \textit{Polyak–Łojasiewicz (PL) inequality}:
\begin{equation}\label{eq:gradient_dominance}
    \mu(J (K)- J^\star) \leq \|\nabla J(K)\|_F^2
     ,\quad
     \forall K\in \mathcal{D}
\end{equation}
where  $J^\star$ is 
the optimal cost value, 
 $\mu$ is a positive constant, and $\mathcal{D}\subseteq \mathcal{K}$ is some suitable region.
This condition 
implies a nonconvex yet strongly convex-like landscape of \cref{eq:LQR_PO_intro} in which first-order methods can converge rapidly, and it is known to hold in both continuous-time and discrete-time LQRs.
The discovery of this benign geometry has also motivated extensions beyond standard LQR, 
such as LQR policy optimization with additional structures \cite{furieri2020learning,zhao2025data,fujinami2025policy,de2025convergence}, 
Kalman filtering \cite{umenberger2022globally,qian2025model}, Linear Quadratic Gaussian (LQG) control \cite{tang2023analysis,li2025policy,zheng2023benign,zheng2022escaping}, and robust control \cite{gravell2020learning,guo2022global,zheng2024benign,watanabe2025policy}.

Despite substantial recent progress, several fundamental questions remain unclear; most notably, the precise relationship between continuous-time and discrete-time LQRs from the perspective of policy optimization is elusive. 
Existing analyses typically treat these two settings separately, leading to differences in assumptions, proof techniques, and the resulting guarantees. \Cref{tab:conditions} lists recent gradient dominance results for the LQR problem.  
First, gradient dominance \cref{eq:gradient_dominance} in discrete-time can hold \textit{globally} over the entire set of stabilizing gains (i.e., we may choose $\mathcal{D} = \mathcal{K}$) \cite{fazel2018global,bu2019lqr}. In contrast, the continuous-time guarantee is
weaker and inherently \textit{semiglobal} \cite{watanabe2025revisiting,de2025remarks}, typically holding only on a sublevel set $\mathcal{D} = \{K\mid J(K)\leq \nu\}$ for some $\nu \geq J^\star$.
Second, the assumptions ensuring \cref{eq:gradient_dominance}  
differ across the existing works
\cite{fazel2018global,bu2019lqr,bu2020policy,mohammadi2019global,fatkhullin2021optimizing,watanabe2025revisiting}. For example, the extent to which one may allow the state weighting matrix $Q$ to be merely positive semidefinite, or the dynamic model $(A,B)$ to be only stabilizable (rather than controllable), is not always made explicit, especially in the discrete-time setting
\cite{fazel2018global,bu2019lqr}. 
Third, a further gap lies in their proof approaches: in continuous time,  gradient dominance 
can be established 
via an exact {convex reformulation} \cite{mohammadi2019global,watanabe2025revisiting}, 
whereas an analogous convex-reformulation-based proof in discrete time has not been developed, aside from preliminary observations in \cite{sun2021learning}.
This discrepancy stems from the different structures of the continuous-time and discrete-time Lyapunov equations: the discrete-time one involves multiplicative couplings that are not directly amenable to the same convexification route. 
While LQR admits a convex reformulation in both time domains after relaxing the Lyapunov equations into inequalities, this additional relaxation step makes a unified analysis elusive.
Consequently, it remains unclear which aspects of gradient dominance are intrinsic to the LQR problem and which are artifacts of the underlying time model.

\vspace{2mm}
\noindent \textbf{Our contributions.} In this paper, we bridge the aforementioned gaps by developing a unified analysis framework of gradient dominance for the continuous-time and discrete-time LQR problems.  
Our main contributions are summarized as follows.

First, we introduce a unified formulation for continuous-time and discrete-time LQRs and
establish a common gradient dominance property in both time settings (\cref{theorem:gradient_dominance}). 
The resulting property involves a non-uniform constant $\mu=\mu_K>0$ that depends on the control gain $K$. 
To obtain a uniform constant,
we further lower-bound $\mu_K$ by eliminating its dependency on $K$, leading to uniform gradient dominance on any compact subset of stabilizing gains (\cref{Corollary:gradient-dominance-compact}) and, under mild additional conditions, a global result for discrete-time LQR (\cref{theorem:global_gradient_dominance}). 
This analysis makes precise how and why discrete-time LQR naturally admits \textit{global} gradient dominance, while the continuous-time counterpart is inherently \textit{semiglobal}. 
As summarized in \cref{tab:conditions},
our \cref{theorem:gradient_dominance} only assumes
stabilizability of $(A,B)$,
detectability of $(Q,A)$ with $Q\succeq0$,
and a positive semidefinite initial covariance $\mathbb{E}[x_0x_0^\tr]=W\succeq0$.
We additionally impose \cref{assumption:X_positive-definite}, which includes $W\succ 0$ as a special case.
One important clarification is that the controllability of $(A,B)$ and strict positive definiteness of $Q$ can be simultaneously relaxed, even in the discrete-time setting, which is not made explicit in \cite{fazel2018global,bu2019lqr}. 
These theoretical findings are illustrated by numerical examples.

Second,
we show that gradient dominance in \cref{theorem:gradient_dominance} naturally emerges from two key inequalities providing, respectively, lower and upper bounds on the optimality gap $J-J^\star$ (\cref{lemma:LQR-two-key-inequalities}).~Such complementary bounds are standard for strongly convex objectives, and we establish that they also hold for both continuous-time and discrete-time LQRs, despite their inherent nonconvexity. 
Our proofs of these bounds are established by a unified convex reformulation of the LQR problems and, in particular, by global properties of the resulting convex program.
A subtle but crucial step is to relax the Lyapunov equations into inequalities and work with an inequality-constrained formulation. This relaxation is necessary in the discrete-time setting and introduces an apparent redundancy in the Lyapunov variables.
We illustrate how to circumvent the redundancy and
establish the two bounds of $J-J^\star$ by exploiting the structure of the convex reformulation.
In particular, the lower bound can be viewed as a quadratic growth property \cite{karimi2016linear,liao2024error} and closely relates to
the classical \textit{completion of squares} technique \cite{zhou1996robust,lancaster1995algebraic}.
For the upper bound, we use a chain rule for nonsmooth functions \cite{rockafellar1998variational} to connect the original LQR cost
with the inequality-constrained convex reformulation.
This enables us to handle the potential nonsmoothness arising from the inequality constraint 
and the elimination of the additional Lyapunov variable (via partial minimization).
Notably, our analysis does not rely on the specific algebraic form of the Lyapunov equations, and it reveals the shared structures underlying gradient dominance in both continuous-time and discrete-time~LQRs.

{
\begin{table}[t]
\centering
\setlength{\abovecaptionskip}{6pt}
\footnotesize
\renewcommand{\arraystretch}{1.3}
\begin{tabular}{crccccccc}
\toprule
Systems & Result& $(A,B)$ 
& $Q$ 
& $W=\mathbb{E}[x_0x_0^\tr]$ &
 Constant $\mu$
 &$\substack{\text{Additional} \\\text{assumptions}}$  \\
\midrule
\multirow{3}{*}{Continuous} & 
$ \substack{\displaystyle
    \;\;\text{Mohammadi et al., 2019 \cite{mohammadi2019global}} \\\displaystyle
      \text{Fatkhullin \& Polyak, 2021\cite{fatkhullin2021optimizing}}
      }$
 & stabilizable 
 & 
 $Q \succ 0$ & $W \succ 0$ &Semiglobal  &No\\
& Bu et al., 2020 \cite{bu2020policy} & stabilizable 
& $Q \succeq 0$ & $W \succ 0$ &Semiglobal &No\\
&Watanabe \& Zheng, 2025 \cite{watanabe2025revisiting}& stabilizable 
& $Q \succeq 0$ & $W \succeq 0$&Semiglobal& Yes\\
\hdashline
\multirow{3}{*}{Discrete}& Fazel et al., 2018 \cite{fazel2018global} & -- 
& $Q \succ 0$ & $W \succeq 0$ &Depends on $K$ &
$\substack{\text{$J(K)$} \\\text{is finite}}$
\\
&Bu et al., 2019 \cite{bu2019lqr} & controllable 
& $Q \succeq 0$ & $W \succ 0$ &Global &No\\
& \cref{theorem:global_gradient_dominance} & stabilizable 
& $Q \succeq 0$ & $W \succ 0$&Global& No\\
\hdashline
\multirow{2}{*}{Unified} & \cref{theorem:gradient_dominance} & stabilizable 
& $Q \succeq 0$ & $W\succeq 0$ & Depends on $K$ &Yes\\
&\cref{Corollary:gradient-dominance-compact} & stabilizable 
& $Q \succeq 0$ & $W\succeq 0$ & Semiglobal &Yes\\
\bottomrule
\end{tabular}
\caption{
Gradient dominance in continuous-time and discrete-time LQRs: 
i) all the results employing $Q\succeq0$ require the detectability of $(Q^{1/2},A)$;
ii) the column ``Constant $\mu$''
corresponds to $\mu>0$ in \cref{eq:gradient_dominance}, where
 ``Semiglobal'' (resp. ``Global'') means that  
\cref{eq:gradient_dominance}  holds with a uniform constant
$\mu >0$ over a compact sublevel set (resp. over the entire domain $\mathcal{K}$).
If $\mu$ takes a different value for each $K$, we write ``Depends on $K$'';
iii)
our \cref{theorem:gradient_dominance,Corollary:gradient-dominance-compact} require  \cref{assumption:X_positive-definite}, which covers $W\succ0$ as a special case  (see \cref{section:gradient_dominance});
iv)
The assumption 
``$J(K)$ is finite''
 means that gradient dominance \cref{eq:gradient_dominance} in
\cite{fazel2018global} holds for $K$ with a finite cost value. 
}
\label{tab:conditions}
\vspace{-2mm}
\end{table}}

\vspace{2mm}
\noindent \textbf{Notation and paper organization.} 
We denote the set of positive (semi)definite matrices by
$\mathbb{S}^n_{++}$ ($\mathbb{S}^n_{+}$).
Given $M_1, M_2\in\mathbb{S}^n$, we use $M_1\prec (\preceq) M_2$ and $M_1\succ (\succeq) M_2$ when $M_1-M_2$ is negative (semi)definite and positive (semi)definite, respectively.
Given $U, V \in \mathbb{R}^{p \times q}$, we write their inner product as $\langle U,V\rangle=\Tr(U^\tr V)$, where $\Tr$ denotes the trace of a square matrix. Let $X \in \mathbb{S}^n$, and we denote its minimum and maximum eigenvalues as $\lambda_{\min}(X)$ and $\lambda_{\max}(X)$, respectively. 

The remainder of this paper is organized as follows.
\cref{section:problem_formulation} presents the problem formulation of the LQR problems.
Then, we state our main results in \cref{section:gradient_dominance}, and present its proof in \cref{section:proof_PL}.
In
\cref{section:proof_partialmin_chain-rule}, we provide a proof of a supporting lemma for the main result.
Finally, \cref{section:conclusion} concludes the paper.
Additional proofs and discussions are provided in \cref{appendix:backgrounds_sec2-3,appendix:proof}. 

%% file: sections/section-2.tex
\vspace{-1mm}

\section{Problem formulation}\label{section:problem_formulation}
\vspace{-0.5mm}
Here, we present the policy optimization formulation of continuous- and discrete-time LQR problems. 

\subsection{Continuous-time and discrete-time LQR problems}\label{subsubsec:discrete-time_LQR}

\textbf{Continuous-time LQR.} 
Consider a continuous-time linear time-invariant (LTI) system:
\begin{equation}\label{eq:dynamic_continuous}
    \dot{x}(t)=A x(t)+B u(t),
\end{equation}
where $x(t) \in \mathbb{R}^n$ is the state vector and $u(t)\in \mathbb{R}^m$ is the input vector. Let $x_0 \in \mathbb{R}^n$ be the initial state at $t = 0$, and we assume $x_0$ is a random variable with zero mean and covariance  
$W = \mathbb{E}\left[x_0x_0^\tr \right] \in \mathbb{S}_{+}^n$.     
We aim to find a control input $u(t)$ to minimize the quadratic cost  
\begin{equation*}
\mathfrak{J}_{\ct}:= 
\mathbb{E}
\left[
\int_{0}^\infty 
\left(x(t)^\tr Qx(t) + u(t)^\tr Ru(t)\right)dt
\right],  
\end{equation*}
where $Q \succeq 0$ and $R\succ 0$ are performance weight matrices, and the expectation above is taken with respect to the random initial state. 
The continuous-time infinite-horizon LQR problem is then formulated as 
\begin{equation} \label{eq:LQR_continuous-time}
    \begin{aligned}
        \min_{u(t)} \quad & \mathfrak{J}_{\ct} \quad
        \text{subject to} \quad  ~\eqref{eq:dynamic_continuous},
    \end{aligned}
\end{equation}
where the input $u(t)$ at time $t$ is allowed to utilize all the past state observation $x(\tau)$ with $\tau \leq t$. 

\vspace{3pt}

\noindent \textbf{Discrete-time LQR.} 
Let us formulate the discrete-time LQR.
Consider a discrete-time LTI system:
\begin{equation}\label{eq:dynamic_discrete-time}
    x_{t+1}=A x_t+B u_t,
\end{equation}
where $x_t \in \mathbb{R}^n$ is the state vector and $u_t\in \mathbb{R}^m$ is the input vector. Let $x_0 \in \mathbb{R}^n$ be the initial state at $t = 0$, and we also assume $x_0$ is a random variable with zero mean and covariance  
$W = \mathbb{E}\left[x_0x_0^\tr \right] \in \mathbb{S}_{+}^n$.     
With two weight matrices $Q \succeq 0$ and $R\succ 0$, the cost in discrete time is given as
\begin{equation*}
\mathfrak{J}_{\dt}:= 
\mathbb{E}
\left[
\sum_{t=0}^\infty 
\left(x_t^\tr Qx_t + u_t^\tr Ru_t\right)
\right],  
\end{equation*}
where the expectation above is taken with respect to the random initial state. 
The discrete-time infinite-horizon LQR problem is formulated as
\begin{equation} \label{eq:LQR_discrete-time}
    \begin{aligned}
        \min_{u_0, u_1, \ldots, u_t, \ldots} \quad & \mathfrak{J}_{\dt} \quad
        \text{subject to} \quad  ~\eqref{eq:dynamic_discrete-time},
    \end{aligned}
\end{equation}
where the input $u_t$ at time $t$ may depend on all past state observation $x_k$, $k=0,\ldots,t$. 

\vspace{3pt}

\noindent \textbf{Globally optimal solutions.}
It is well known that the globally optimal solution to \cref{eq:LQR_continuous-time,eq:LQR_discrete-time} can be obtained by solving an algebraic Riccati equation. 
We make the following standard assumption. 

\begin{assumption} \label{assumption:stabilizable}
    The performance weights satisfy 
    $Q\in\mathbb{S}_{+}^n\setminus\{0\}, 
    R\in\mathbb{S}^m_{++}$, and 
    the covariance matrix $W$ satisfies $W\in \mathbb{S}_{+}^n\setminus\{0\}$. The pair
    $(A, B)$ is stabilizable
    and $(Q^{1/2}, A)$ is detectable. 
\end{assumption}

Note that the stabilizability and detectability for continuous-time and discrete-time systems indicate essentially the same concepts, but their mathematical characterizations slightly differ from each other; see \cite[Sec. 3.2]{zhou1996robust} and \cite[Appendix A]{chen2012optimal} for details.  
The following two theorems are well-known (see e.g., \cite[Th. 14.2]{zhou1996robust} and \cite[Th. 6.3.2]{chen2012optimal}).
\begin{theorem}[Optimal input in continuous time]
\label{theorem:LQR_solution_continuous-time}
Consider the continuous-time LQR \cref{eq:LQR_continuous-time}. Under \cref{assumption:stabilizable}, the globally optimal input is unique, given by the stabilizing state feedback
policy 
\begin{subequations} \label{eq:LQR-gain-Ct}
\begin{equation}
    u(t) =
    K^\star_{\ct} x(t), \qquad \forall t \geq 0
\end{equation}
where the optimal gain is $ K^\star_{\ct} =
    -R^{-1}
    B^\tr P^\star $ with $P^\star$ being the unique positive semidefinite solution~to
\begin{equation}\label{eq:CARE}
   A^{\tr} P+P A+Q-P B R^{-1} B^{\tr} P =0. 
\end{equation}
\end{subequations}
\end{theorem}
\begin{theorem}[Optimal input in discrete time]
\label{theorem:LQR_solution_discrete-time}
Consider the discrete-time LQR \cref{eq:LQR_discrete-time}. Under \cref{assumption:stabilizable}, the globally optimal input is unique, given by the stabilizing state feedback
policy 
\begin{subequations} \label{eq:LQR-gain-DT}
\begin{equation}
    u_t =
    K^\star_{\dt} x_t,  \qquad t = 0, 1, 2, \ldots, 
\end{equation}
where the optimal gain is $ K^\star_{\dt} =
    -(R+B^\tr P^\star B)^{-1}
    B^\tr P^\star A$ with $P^\star$ being the unique positive semidefinite solution~to
\begin{equation}\label{eq:DARE}
    P=A^\tr PA
    +Q-A^\tr PB(R+B^\tr P B)^{-1}B^\tr P A.
\end{equation}
\end{subequations}
\end{theorem}

\subsection{LQR policy optimization}\label{subsection:LQR_PO}
The optimal LQR solutions in \cref{theorem:LQR_solution_continuous-time,theorem:LQR_solution_discrete-time} are model-based, since solving the Riccati equations in \cref{eq:CARE,eq:DARE} requires the model information $A, B, Q$, and $R$. In the past few years, model-free methods for LQR have received significant attention 
both in continuous time \cite{mohammadi2019global,bu2020policy,fatkhullin2021optimizing,watanabe2025revisiting} and in discrete time \cite{fazel2018global,bu2019lqr,gravell2020learning}, which directly optimizes the LQR policy.

In light of
\Cref{theorem:LQR_solution_continuous-time,theorem:LQR_solution_discrete-time}, we can parameterize the class of policies to be linear, static, and stabilizing, since it is rich enough to contain the optimal policy. 
In the continuous-time case, consider 
\begin{equation*}
    u(t) = Kx(t)\quad \text{with} \quad
    K\in\mathcal{K}_{\ct}
    :=\{K\in \mathbb{R}^{m\times n}\mid
    \operatorname{Re}\left(\lambda_i(A+BK)\right)<0,\,\forall i=1,\ldots,n
    \}.
\end{equation*}
The policy optimization for the continuous-time LQR \eqref{eq:LQR_continuous-time} is 
\begin{equation} \label{eq:LQR-stochastic-noise_continuous}
\begin{aligned}
\min_{K}\ \ &
J_{\ct}(K):=\mathbb{E}\;
\left[
\int_{0}^{\infty}
\left(x(t)^\tr Q x(t) + u(t)^\tr R u(t)\right)dt
\right]
\\
\text{subject to}\ \ & ~\eqref{eq:dynamic_continuous}, \quad u(t) = Kx(t),\quad
K \in \mathcal{K}_{\ct}.
\end{aligned}
\end{equation}
Analogously, in the discrete-time case, we consider 
\begin{equation*}
    u_t = Kx_t
    \quad \text{with} \quad
    K\in\mathcal{K}_{\dt}
    :=\{
    K\in \mathbb{R}^{m\times n}\mid
    \rho(A+BK)<1
    \},
\end{equation*}
where $\rho(\cdot)$ denotes the spectral radius.
The policy optimization for the discrete-time LQR \cref{eq:LQR_discrete-time} is 
\begin{equation} \label{eq:LQR-stochastic-noise_discrete}
    \begin{aligned}
\min_{K}\ \ &
J_{\dt}(K):=\mathbb{E}\;
\left[
\sum_{t=0}^{\infty}
\left(x_t^\tr Q x_t + u_t^\tr R u_t\right)
\right]
\\
\text{subject to}\ \ & ~\eqref{eq:dynamic_discrete-time}, \quad u_t = Kx_t,\quad
K \in \mathcal{K}_{\dt}. 
\end{aligned}
\end{equation}

Both \cref{eq:LQR-stochastic-noise_continuous,eq:LQR-stochastic-noise_discrete} directly search for the optimal feedback policy parameter $K$.  
Here, $J_{\ct}: \mathcal{K}_{\ct} \to \mathbb{R}$
(resp. $J_{\dt}: \mathcal{K}_{\dt} \to \mathbb{R}$)
denotes the LQR cost in continuous time (resp. in discrete time)
under the policy $K \in \mathcal{K}_{\ct}$ (resp.  $K \in \mathcal{K}_{\dt}$). 
We can now simply rewrite \cref{eq:LQR-stochastic-noise_continuous,eq:LQR-stochastic-noise_discrete} as 

\vspace{8pt}

\begin{subequations}
\label{eq:LQR-H2}
\begin{minipage}{0.48\linewidth}
\begin{equation}
\label{eq:LQR-H2-ct}
\min_{K \in \mathcal{K}_{\ct}} \; J_{\ct}(K)
\end{equation}
\end{minipage}
\hfill
\begin{minipage}{0.48\linewidth}
\begin{equation}
\label{eq:LQR-H2-dt}
\min_{K \in \mathcal{K}_{\dt}} \; J_{\dt}(K).
\end{equation}
\end{minipage}
\end{subequations}

\vspace{8pt}
It is well-known that the LQR policy optimization problems \cref{eq:LQR-H2} are nonconvex. Fortunately, the cost functions are known to be continuously differentiable on $\mathcal{K}_\ct$ and $\mathcal{K}_\dt$ \cite{fazel2018global,bu2020policy,fatkhullin2021optimizing,watanabe2025revisiting,mohammadi2019global}. 
Furthermore, the LQR costs satisfy \textit{gradient dominance} under various assumptions (see \Cref{tab:conditions}),
which allows the simple gradient descent $K_{l+1} = K_{l} - \eta \nabla J(K_l)$ with a proper stepsize to enjoy linear convergence
to a global minimizer (see e.g.,  \cite[Lem. 5]{watanabe2025revisiting} and \cite[Th. 1]{hu2023toward}).
In this paper, our main goal is to present a unified analysis for  
gradient dominance in continuous-time and discrete-time LQRs \cref{eq:LQR-H2}. %

%% file: sections/section-3.tex
\section{Unified gradient dominance analysis}\label{section:gradient_dominance}

In this section, we first 
illustrate how to unify the continuous-time and discrete-time LQR problems.
We then present a unified gradient dominance result %
with an additional \cref{assumption:X_positive-definite}.
We finally discuss the implications of \cref{assumption:X_positive-definite}
with a few illustrative examples.

\subsection{A unified formulation for continuous- and discrete-time LQRs}\label{subsection:unif_LQR}

We here introduce a linear operator $\Psi$ to provide a unified formulation for the LQR problems in \cref{eq:LQR-H2}. 
Related unification of continuous-time and discrete-time systems has appeared in the context of LMIs
\cite{peaucelle2000new,iwasaki2005generalized}.
In the following discussion, Lyapunov equations play a fundamental role. For the self-completeness, we summarize their properties in \cref{subsection:Lyapunov_summary}.

We begin with rewriting \cref{eq:LQR-H2-ct}--\cref{eq:LQR-H2-dt} into 
nonconvex programs that involve Lyapunov equations.
In the continuous-time case, for $K \in \mathcal{K}_{\ct}$, the LQR cost can be evaluated as $J_{\ct}(K) = \langle Q+K^\tr RK,X\rangle$ (see e.g., \cite[Sec. 4.3, Sec. 4.6.1]{skelton1997unified} and \cite{bu2020policy,watanabe2025revisiting}), where $X$ is the unique solution to the continuous-time Lyapunov equation
\begin{equation}\label{eq:Lyapunov_eq_continuous}
    (A+BK)X+X(A+BK)^\tr  +W  = 0 
\end{equation}
Thus, \cref{eq:LQR-H2-ct} can be equivalently rewritten as
\begin{equation}\label{eq:LQR_KX_continuous_time}
\begin{aligned}
    \min_{K,\,X}
    &\quad 
    \langle Q+K^\tr RK,X\rangle
    \\
    \text{subject to}&\quad
\cref{eq:Lyapunov_eq_continuous},\quad
    K\in\mathcal{K}_{\ct}.
\end{aligned}
\end{equation}
Analogously, for $K \in \mathcal{K}_{\dt}$, the discrete-time LQR cost can be evaluated as $J_{\dt}(K) = \langle Q+K^\tr RK,X\rangle$ (see e.g., \cite[Sec. 4.6.2]{skelton1997unified} and \cite{bu2019lqr}), where $X$ is the unique solution to the discrete-time Lyapunov equation
\begin{equation}\label{eq:Lyapunov_eq_discrete}
    (A+BK)X(A+BK)^\tr -X +W  = 0. 
\end{equation}
Therefore, the discrete-time LQR \cref{eq:LQR-H2-dt} can be represented as
\begin{equation}\label{eq:LQR_KX_discrete_time}
\begin{aligned}
    \min_{K,\,X}
    &\quad 
    \langle Q+K^\tr RK,X\rangle
    \\
    \text{subject to}&\quad
\cref{eq:Lyapunov_eq_discrete},\quad
    K\in\mathcal{K}_{\dt}.
\end{aligned}
\end{equation}

One main difference in \Cref{eq:LQR_KX_continuous_time,eq:LQR_KX_discrete_time} lies in the form of the associated Lyapunov equation. Motivated by this observation, we define two linear operators corresponding to the continuous-time and discrete-time settings, respectively: 
\begin{align*}
    \Psi_{\ct}
    \left(
    \begin{bmatrix}
        F & G\\
        G^\tr& H
    \end{bmatrix}
    \right)
    =  G+G^\tr,\qquad
    \Psi_{\dt}
    \left(
    \begin{bmatrix}
        F & G\\
        G^\tr& H
    \end{bmatrix}
    \right)
    =F-H,
\end{align*}
where $F, G, H$ are real matrices of compatible dimensions. 
Using these two operators, both the continuous-time and discrete-time LQR problems in \cref{eq:LQR-H2}  can be written in the unified form 
\begin{subequations}\label{eq:LQR_unif}
\begin{align}
    \min_{K,\,X}
    &\quad 
    \langle Q+K^\tr RK,X\rangle
    \\
    \label{eq:LQR_unif_Lyapunov}
    \text{subject to}&\quad
    \Psi
   \left(
   \begin{bmatrix}
       A+BK\\
       I
   \end{bmatrix}X
   \begin{bmatrix}
       A+BK\\
       I
   \end{bmatrix}^\tr
   \right)
    +W = 0
    ,\quad
    K\in\mathcal{K}. 
\end{align}
\end{subequations}
Here, $\Psi = \Psi_{\ct}$ and $\mathcal{K}=\mathcal{K}_\ct$ correspond to the continuous-time case, while $\Psi = \Psi_{\dt}$ and $\mathcal{K}=\mathcal{K}_\dt$ correspond to the discrete-time case.  

We summarize the discussion above as a lemma below.
This unified representation \cref{eq:LQR_unif} plays an important role in developing a common proof framework for the gradient dominance property of continuous-time and discrete-time LQRs.
\begin{lemma} \label{lemma:LQR-unification}
    Under \Cref{assumption:stabilizable}, the continuous-time LQR \cref{eq:LQR-H2-ct} and problem \cref{eq:LQR_unif} with $\Psi = \Psi_{\ct}$ and $\mathcal{K}=\mathcal{K}_\ct$ are equivalent in the sense that they have the same optimal gain and the same optimal value. The discrete-time LQR \cref{eq:LQR-H2-dt} and problem \cref{eq:LQR_unif} with $\Psi = \Psi_{\dt}$ and $\mathcal{K}=\mathcal{K}_\dt$ are also equivalent. 
\end{lemma}

\subsection{Gradient dominance in continuous- and discrete-time LQRs}\label{subsection:PL_nonuniform}

For each stabilizing gain $K \in \mathcal{K}$, the Lyapunov equation in \cref{eq:LQR_unif_Lyapunov} admits a unique solution $X$. Consequently, the Lyapunov variable can be viewed as an implicit function of the policy parameter $K$, which we denote by $X_K$. With this elimination of $X$, we can rewrite the unified LQR \cref{eq:LQR_unif}  as a policy optimization problem over the set of stabilizing gains
\begin{equation}\label{eq:LQR_PO_unif}
    J^\star =\min_{K\in\mathcal{K}}\; J(K),
\end{equation}
where  
$ J(K) = \langle Q+K^\tr RK, X_K\rangle, 
$ 
and $X_K$ is the unique solution to the Lyapunov equation in \cref{eq:LQR_unif_Lyapunov} associated with $K$. 
This is simply a reformulation of the continuous-time and discrete-time LQRs in \cref{eq:LQR-H2}, in which the LQR cost is evaluated via the corresponding Lyapunov equation.

Beyond its role in evaluating the LQR cost, the Lyapunov variable $X_K$
 plays a central role in the analysis of the optimization landscape. In both continuous-time and discrete-time settings, it is well known that $X_K$
 is positive semidefinite for any stabilizing gain $K \in \mathcal{K}$. To develop gradient dominance properties, we impose the following technical assumption.
\begin{assumption}\label{assumption:X_positive-definite}
Let $X_K$ be a solution to the Lyapunov equation \cref{eq:LQR_unif_Lyapunov}
corresponding to a policy parameter $K \in \mathbb{R}^{m \times n}$. We assume that 
$   K\in \mathcal{K}$ if and only if $
   X_K$ is positive definite. 
\end{assumption}

This assumption allows us to derive equivalent convex reformulations of the continuous-time and discrete-time LQRs \cref{eq:LQR_unif}. 
We defer the details to \Cref{subsection:convex_reformulation}.
In addition, 
\Cref{assumption:X_positive-definite} is not a strict requirement 
and has further control-theoretic implications.
We provide these discussions in \cref{subsection:assumption-2}. 
Under \cref{assumption:stabilizable,assumption:X_positive-definite},
the unified LQR cost function in \cref{eq:LQR_PO_unif} satisfies
a gradient dominance property with a constant $\mu_K$ depending on $K\in\mathcal{K}$.
The result applies uniformly to both continuous-time and discrete-time systems.
\begin{theorem}[Unified gradient dominance with a non-uniform constant]\label{theorem:gradient_dominance}
Consider the LQR policy optimization problem \cref{eq:LQR_PO_unif}.
Under 
\cref{assumption:stabilizable,assumption:X_positive-definite},
the following inequality holds 
\begin{equation}\label{eq:gradient_dominance_thm}
         \mu_K
         \left(J(K)-J^\star \right)
     \leq 
     \|\nabla J(K)\|_F^2, \quad \forall \, K\in \mathcal{K},
\end{equation}
where the constant $\mu_K >0$ is given by 
\begin{equation} \label{eq:gradient-domiance-constant}
    \mu_K = 
\frac{\lambda_\mathrm{min}(R)
     \lambda_\mathrm{min}(X_K)^3}
     {\lambda_\mathrm{max}(X^\star)^2}. 
\end{equation}
Here, $X_K$  denotes the solution to the Lyapunov equation \cref{eq:LQR_unif_Lyapunov} associated with $K$, and $X^\star$ denotes the corresponding Lyapunov solution for the optimal LQR gain $K^\star$.
\end{theorem}
Our proof is based on a convex-lifting argument and leverages an exact convex reformulation of \cref{eq:LQR_unif}. The technical details are presented in \cref{section:proof_PL}.

It is known that gradient dominance holds for continuous-time and discrete-time LQRs when analyzed separately \cite{fazel2018global,mohammadi2019global,fatkhullin2021optimizing,bu2019lqr,bu2020policy}. Existing analyses, however, rely on proof techniques tailored to the specific time-domain setting, and the resulting gradient dominance constants take different forms in the continuous-time and discrete-time cases. 
In contrast, the gradient dominance property established in \Cref{theorem:gradient_dominance} applies uniformly to both continuous-time and discrete-time. The corresponding constant in \cref{eq:gradient-domiance-constant} admits the same functional form in both settings. Moreover, the proof proceeds by an identical convex-lifting argument based on the unified Lyapunov representation. This unified perspective not only streamlines the analysis and clarifies the underlying LQR landscapes but also reveals a deeper structural symmetry between continuous-time and discrete-time formulations.

We note that the gradient dominance constant in \cref{eq:gradient-domiance-constant} is strictly positive but depends on the specific policy. If there exists a uniform lower bound on the minimum eigenvalue of $X_K$, then a uniform gradient dominance constant follows immediately. Deriving such a bound typically requires additional~assumptions, and the continuous-time and discrete-time settings may necessitate different arguments. 
Nevertheless, we have the following corollary that applies to both 
settings.

\begin{corollary}[Uniform constant $\mu$ over a compact set] \label{Corollary:gradient-dominance-compact}
    Consider the LQR policy optimization problem \cref{eq:LQR_PO_unif}.
Suppose  
\cref{assumption:stabilizable,assumption:X_positive-definite} hold. Then, for any compact subset $\tilde{\mathcal{K}} \subseteq \mathcal{K}$, there exists a positive constant $\mu >0$ such that %
\begin{equation*}
         \mu
         \left(J(K)-J^\star \right)
     \leq 
     \|\nabla J(K)\|_F^2, \quad \forall \, K\in \tilde{\mathcal{K}}. 
\end{equation*}
\end{corollary}
\begin{proof}
    The solution $X_K$ to the Lyapunov equation \cref{eq:LQR_unif_Lyapunov} is a rational (and thus continuous) function of $K$ over $\mathcal{K}$. Recall that $\lambda_\mathrm{min}(X_K) > 0$ for any $K \in \tilde{\mathcal{K}}$. We can take 
    $$
    \lambda_{0} = \min_{K \in \tilde{\mathcal{K}}}  \lambda_\mathrm{min}(X_K),
    $$
    which is strictly positive due to the continuity and the compactness. We complete the proof by taking $\mu = {\lambda_\mathrm{min}(R)
     \lambda_{0}^3}/
     {\lambda_\mathrm{max}(X^\star)^2} > 0$ in \Cref{eq:gradient-domiance-constant}.
\end{proof}

\Cref{Corollary:gradient-dominance-compact} establishes gradient dominance only over compact subsets of stabilizing gains.
While the compactness assumption may appear restrictive, it is already sufficient to guarantee linear convergence of the standard gradient descent for a broad class of continuous-time and discrete-time LQR problems (see \cite[Th. 1]{hu2023toward} and \cite[Lem. 5]{watanabe2025revisiting}). In particular, if a sublevel set of the cost, 
$\mathcal{K}_\nu :=\{K \in \mathcal{K} \mid J(K) \leq \nu\},$  
is compact, 
then \Cref{Corollary:gradient-dominance-compact} naturally holds over $\mathcal{K}_\nu$. Note that gradient descent with an appropriate step size ensures that all iterates remain within the initial sublevel set (see \cite[Lem. 5]{watanabe2025revisiting} and \cite[Th. 1]{hu2023toward}), so \cref{eq:gradient_dominance_thm} holds for all iterations. 

For continuous-time LQR, a uniform gradient dominance constant may fail to exist over an unbounded set of stabilizing gains, even under stronger assumptions such as controllability of $(A, B)$ and positive definiteness of $Q$ and $W$ \cite{watanabe2025revisiting,de2025remarks}. An explicit example can be found in \cite[Example 3]{watanabe2025revisiting}. This is because we may not be able to uniformly bound the minimum eigenvalue of $X_K$ over an unbounded region. 
In this sense, \Cref{Corollary:gradient-dominance-compact} represents the strongest form of gradient dominance that can be expected for continuous-time LQR.
 
In contrast, for discrete-time LQR, a global gradient dominance property can be established under mild conditions since we can uniformly bound the minimum eigenvalue of $X_K$. 

\begin{corollary}[Global gradient dominance for discrete-time systems]\label{theorem:global_gradient_dominance}
    Consider the discrete-time LQR policy optimization problem \cref{eq:LQR-H2-dt}.
Suppose  
\cref{assumption:stabilizable} holds. If $W \succ 0$, then there exists a positive constant $\mu >0$ such that 
\begin{equation}\label{eq:gradient_dominance_thm-global}
         \mu
         \left(J_\dt(K)-J^\star \right)
     \leq 
     \|\nabla J_\dt(K)\|_F^2, \quad \forall \, K\in {\mathcal{K}_\dt}. 
\end{equation}
\end{corollary}
\begin{proof}
    When $W \succ 0$, \Cref{assumption:X_positive-definite} holds naturally; see \Cref{proposition:X_positive} in \Cref{subsection:assumption-2}. Furthermore, the solution to the discrete-time Lyapunov equation \cref{eq:Lyapunov_eq_discrete} can be written as 
 $   X_K=\sum_{t=0}^\infty (A+BK)^t W\left((A+BK)^\tr\right)^t$ for any $K \in \mathcal{K}_{\dt}.$ 
Thus, we know $X_K\succeq W\succ 0$ for any $K \in \mathcal{K}_\dt$. Taking $\mu = {\lambda_\mathrm{min}(R)
     \lambda_\mathrm{min}(W)^3}/
     {\lambda_\mathrm{max}(X^\star)^2} > 0$ in \Cref{eq:gradient-domiance-constant} completes the proof. 
\end{proof}

The positive definiteness of $W$ has been widely used to establish the global gradient dominance in the discrete-time LQR; see for example \cite{fazel2018global,bu2019lqr}. In \Cref{theorem:global_gradient_dominance}, we only require \Cref{assumption:stabilizable}, i.e., $(A,B)$ is stabilizable and $(Q^{1/2},A)$ is detectable. This slightly generalizes existing results \cite{fazel2018global,bu2019lqr}  by allowing the weighting matrix $Q$ to be positive semidefinite rather than strictly positive definite.

\subsection{Discussions on \Cref{assumption:X_positive-definite} and examples} \label{subsection:assumption-2}

In addition to the standard \Cref{assumption:stabilizable}, the gradient dominance result in \Cref{theorem:gradient_dominance} requires an additional \cref{assumption:X_positive-definite} involving the Lyapunov variable $X_K$. 
We here provide further control-theoretic implications and present several examples for illustration. %

It is well-known that the Lyapunov variable $X_K$ in \cref{eq:LQR_unif_Lyapunov} can be interpreted as the accumulated covariance of the system state. In particular, consider the closed-loop dynamics with a feedback gain $K \in \mathcal{K}$: 
$\dot x(t) = (A+BK)x(t)$ in continuous time
or
$x_{t+1} = (A+BK)x_t$ in discrete time. The initial state is random with zero mean and covariance $\mathbb{E}[x_0 x_0^\tr] = W \succeq 0$. Since $A+BK$ is stable, the instantaneous covariance $\mathbb{E}[x(t)x(t)^\tr]$ (or $\mathbb{E}[x_t x_t^\tr]$) converges to zero and the accumulated covariance,
\begin{equation} \label{eq:second-order-moment-state} 
X_K := \int_{0}^{\infty} \mathbb{E}[x(t)x(t)^\tr]\,dt, \quad \text{(CT)} \qquad
X_K := \sum_{t=0}^{\infty} \mathbb{E}[x_t x_t^\tr], \quad \text{(DT)}
\end{equation}
is well defined and finite, which coincides with the unique solution to the Lyapunov equation~\cref{eq:LQR_unif_Lyapunov}.  

In the sense of \cref{eq:second-order-moment-state}, $X_K$ quantifies the total state covariance accumulated over time due to the initial randomness. The condition $X_K \succ 0$ implies that this accumulated covariance \cref{eq:second-order-moment-state} is nondegenerate in every direction in the state space. 
With this interpretation, we present the following sufficient conditions for \cref{assumption:X_positive-definite}. 
\begin{proposition}\label{proposition:X_positive}
In both continuous-time and discrete-time settings, suppose \cref{assumption:stabilizable} holds.
Then, \cref{assumption:X_positive-definite} is satisfied if one of the following conditions holds:
\begin{enumerate}
    \item The covariance matrix $W$ is positive definite.
    \item The pair $(A,B)$ is controllable, and $\operatorname{Im}(B)\subseteq \operatorname{Im}(W)$ holds.
\end{enumerate}
\end{proposition}
\begin{proof}
If either condition 1) or condition 2) holds, we know that 
$(A+BK,W^{1/2})$ is controllable for any $K\in\mathbb{R}^{m\times n}$.
Thus, it suffices to show that
when $(A+BK,W^{1/2})$ is controllable,
we have $A+BK$ is stable (i.e.,
$K\in\mathcal{K}$) if and only if the Lyapunov equation \cref{eq:LQR_unif_Lyapunov} admits 
a unique positive definite solution $X_K$. This equivalence is standard in classical control theory; see \cite[Th. 5.3.2. (b)]{lancaster1995algebraic} for the continuous-time case and \cite[Th. 5.3.5. (b)]{lancaster1995algebraic} for the discrete-time case. We present further details in \cref{subsection:proof_X_PD}. 
\end{proof}

As mentioned after \Cref{theorem:global_gradient_dominance}, the positive definiteness of $W$ is the most common condition to establish gradient dominance in both continuous-time \cite{mohammadi2019global,bu2020policy,fatkhullin2021optimizing} and discrete-time \cite{fazel2018global,bu2019lqr} LQRs. The second condition in \Cref{proposition:X_positive} in fact allows $W$ to be positive semidefinite, but the accumulated covariance \cref{eq:second-order-moment-state} is still guaranteed to be positive definite thanks to the controllability of $(A+BK, W^{1/2})$. Thus, \Cref{theorem:gradient_dominance} holds for some LQR instances even when $W \succeq 0$. The case of $W \succeq 0$ has been further discussed in detail for continuous-time LQR in our prior work \cite{watanabe2025revisiting}.

We here illustrate the two conditions in \Cref{proposition:X_positive} for the discrete-time LQR below.

\begin{figure}
\begin{subfigure}{0.48\columnwidth}
  \centering
\includegraphics[width=0.7\columnwidth]{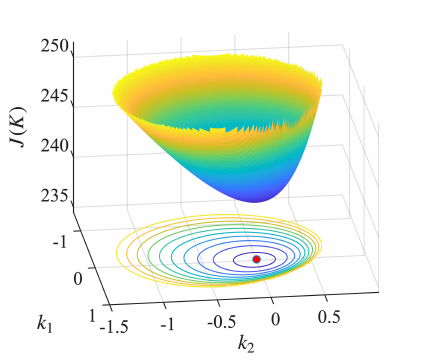}
\caption{
\cref{example:PL_Q-PSD_AB-stabilizable}
}
\label{fig:PL_Q-PSD_AB-stabilizable}
\end{subfigure}
\begin{subfigure}{0.48\columnwidth}
  \centering
\includegraphics[width=0.7\columnwidth]{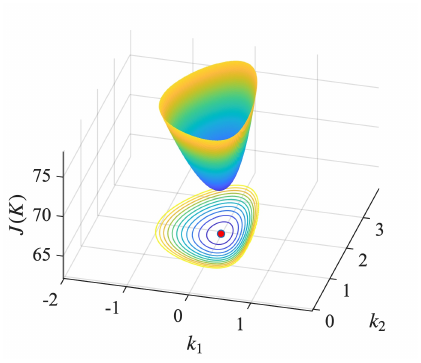}
\caption{
\cref{example:PL_WQ-PSD}
}
\label{fig:PL_WQ-PSD}
\end{subfigure}%
\caption{
Gradient dominance in the discrete-time LQR:
(a) Landscape in \cref{example:PL_Q-PSD_AB-stabilizable}, where 
$W\succ0$ but $Q\nsucc 0$ and $(A,B)$ is only stabilizable;
(b) Landscape in \cref{example:PL_WQ-PSD},
where $(A,B)$ is controllable, and both $Q$ and $W$ are only positive semidefinite. 
In both figures, the red dots  denote the optimal gain $K^\star$.
}
\label{fig:fig1}
\end{figure}  

\begin{example}[Global gradient dominance with $W\succ 0$]\label{example:PL_Q-PSD_AB-stabilizable}
    Consider the discrete-time system with
    \begin{align*}
        A = \begin{bmatrix}
            0.9 & 0.01\\
            0.01 & 0.9
        \end{bmatrix},\quad
        B = \begin{bmatrix}
            0.1\\0.1
        \end{bmatrix},\quad
        Q = \begin{bmatrix}
            1 & -1\\
            -1 & 1
        \end{bmatrix}\succeq0,\quad
        R=0.1, \quad
        W = 25I_2.
    \end{align*}
    In this example, $(A,B)$ is only stabilizable but is uncontrollable.
    The matrix $Q$ is only positive semidefinite, and we have 
    $W\succ0$.
    We plot
    the LQR cost $J$ in
    \cref{fig:PL_Q-PSD_AB-stabilizable}, which shows a strongly-convex-like landscape, 
    as expected from the global gradient dominance result in \cref{theorem:global_gradient_dominance}. \hfill$\square$
\end{example}

\begin{example}[Local gradient dominance with $Q\nsucc0$ and $W\nsucc 0$]\label{example:PL_WQ-PSD}
In discrete-time, consider 
    \begin{align*}
   A = \begin{bmatrix}
        0.9 & 0.01\\
        0.01 &- 0.9
    \end{bmatrix},\quad
    B=
    \begin{bmatrix}
      0.1\\0.1  
    \end{bmatrix},\quad
    Q =\begin{bmatrix}
        0 &0\\
        0&1
    \end{bmatrix}
    ,\quad 
    R=1,\quad
    W 
    =25\begin{bmatrix}
        1 & 1\\
        1 & 1
    \end{bmatrix}.
    \end{align*}
    Now $(A,B)$ is controllable, and the second sufficient condition of \cref{proposition:X_positive} holds,
    which ensures \cref{assumption:X_positive-definite}.
    Moreover, $(Q^{1/2},A)$ is detectable, but 
    $Q$ is not positive definite.
    We present the plot of
    $J$ in
    \cref{fig:PL_WQ-PSD}.
    We still observe its benign landscape, where gradient dominance indeed holds at least locally, as expected from \cref{Corollary:gradient-dominance-compact}. \hfill$\square$
\end{example}

\cref{assumption:X_positive-definite} also implies uniqueness of the optimal LQR gain in both continuous-time and discrete-time settings. The proof is based on stationary-point characterizations, and we provide some details in \cref{subsection:proof_uniqueness}. 

\begin{proposition}\label{proposition:uniqueness}
    Consider the LQR policy optimization problem \cref{eq:LQR_PO_unif}. The following statements hold in both continuous-time and discrete-time systems.
    \begin{enumerate}
        \item Under \Cref{assumption:stabilizable}, the problem \cref{eq:LQR_PO_unif} admits at least one optimal gain, though the optimal solution may be non-unique.
\item If, in addition, \Cref{assumption:X_positive-definite} holds, then the optimal gain 
$K^\star$  is unique. 
    \end{enumerate}
\end{proposition}

When only \Cref{assumption:stabilizable} is imposed, the optimal gain obtained from the Riccati equations in \Cref{theorem:LQR_solution_continuous-time,theorem:LQR_solution_discrete-time} corresponds to one particular optimal solution to \cref{eq:LQR_PO_unif}, but additional optimal gains may exist. In particular, if $X_K$ in \cref{eq:second-order-moment-state} is only positive semidefinite, there exist state directions that contribute zero accumulated energy to the cost. Consequently, perturbations of the feedback gain along these directions do not affect the LQR objective, leading to flat directions in the optimization landscape and possibly multiple stabilizing gains achieving the same optimal value. 

In contrast, \Cref{assumption:X_positive-definite} eliminates such degeneracies by ensuring that all directions contribute strictly to the cost. This restores strict curvature around the optimum and guarantees that the first-order stationarity admits a unique solution, thereby ensuring the uniqueness of the optimal LQR gain in both continuous- and discrete-time. %
We present two explicit LQR instances below.

\begin{figure}
\begin{subfigure}{0.48\columnwidth}
  \centering
\includegraphics[width=0.7\columnwidth]{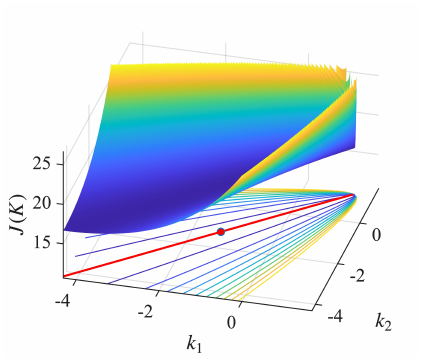}
\caption{
Continuous-time
(\cref{example:PL_fail_ct})
}
\label{fig:PL_fail_ct}
\end{subfigure}
\begin{subfigure}{0.48\columnwidth}
  \centering
\includegraphics[width=0.7\columnwidth]{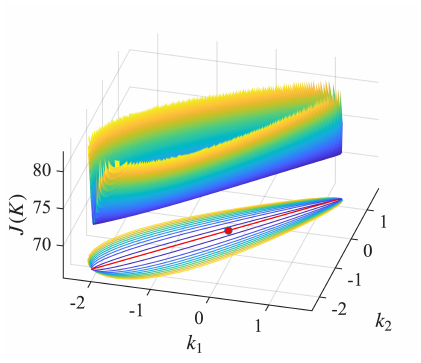}
\caption{
Discrete-time
(\cref{example:PL_fail_dt})
}
\label{fig:PL_fail_dt}
\end{subfigure}%
\caption{
Non-unique optimal LQR gains when \cref{assumption:X_positive-definite} fails:
(a) The continuous-time case in \cref{example:PL_fail_ct};
(b) The discrete-time case in \cref{example:PL_fail_dt}. 
In both cases, the optimal value of $J$ is achieved at any points over the red line, indicating the non-uniqueness of the optimal LQR gains.
}
\label{fig:saddle_point}
\end{figure}

\begin{example}[Non-unique optimal LQR gains in continuous time]\label{example:PL_fail_ct}
    Consider the following problem data for the continuous-time LQR \cref{eq:LQR-H2-ct}:
    \begin{align*}
    A = -\begin{bmatrix}
        2 & 0.5\\
        0.5 &2
    \end{bmatrix},\quad
    B=
    \begin{bmatrix}
      1\\1  
    \end{bmatrix},\quad
    Q =I_2
    ,\quad
    R=1,\quad
     W 
    =25\begin{bmatrix}
        1 & -1\\
        -1 & 1
    \end{bmatrix}.
    \end{align*}
 It is clear that $(A,B)$ is stabilizable, $(Q^{1/2},A)$ is observable, $W\succeq0, R\succ 0$. Thus  \cref{assumption:stabilizable} is satisfied. However, 
    \cref{assumption:X_positive-definite} fails because $X_K$ is only positive semidefinite in \Cref{eq:LQR_unif_Lyapunov} (since $(A,B)$ is stabilizable but not controllable). In this case, the optimal LQR gain to \cref{eq:LQR-H2-ct} is indeed not unique. \cref{fig:PL_fail_ct} shows that there exists a flat region corresponding to all optimal gains. The gain obtained from \Cref{theorem:LQR_solution_continuous-time} is one particular optimal solution to \cref{eq:LQR_PO_unif}, marked by the red point in \cref{fig:PL_fail_ct}. 
    \hfill$\square$
\end{example}
\begin{example}[Non-unique optimal LQR gains in discrete time]\label{example:PL_fail_dt}
    We apply the Euler-discretization to the problem data in \cref{example:PL_fail_ct}.
    Namely, 
    for the discretization step $\Delta t$,
    consider
    \begin{align*}
    A =I_2- \Delta t \begin{bmatrix}
        2 & 0.5\\
        0.5 &2
    \end{bmatrix},\quad
    B=\Delta t \begin{bmatrix}
        1\\1
    \end{bmatrix}
    \end{align*}
    with the same $Q$, $W$, and $R$.
    In this case,
    \Cref{assumption:stabilizable} holds but \cref{assumption:X_positive-definite} still fails since $X_K$ is only positive semidefinite in \Cref{eq:LQR_unif_Lyapunov} (again, this is because  
    $(A,B)$ is stabilizable but uncontrollable). 
    \cref{fig:PL_fail_dt} shows the discrete-time LQR landscape for 
    the case of $\Delta t=0.3$. We indeed observe the non-uniqueness of its optimal solutions to \cref{eq:LQR_PO_unif} in this discrete-time setting.
    \hfill$\square$
\end{example}

%% file: sections/section-4.tex
 \section{A unified proof based on convex reformulation}\label{section:proof_PL}

This section presents the proof of \cref{theorem:gradient_dominance}
 by leveraging a hidden convex structure of the LQR.
We first give a proof sketch using 
two key inequalities that lower- and upper-bound $J-J^\star$ in \cref{subsection:proof-sketch}. 
We then introduce a convex reformulation with a (non-standard) change of variables in \cref{subsection:convex_reformulation}.
We finally prove the key inequalities using the convex reformulation in \cref{subsection:proof-lemma-1}.

\subsection{Proof of \Cref{theorem:gradient_dominance} via two key inequalities}\label{subsection:proof-sketch}

Our proof strategy leverages the well-known hidden convexity in LQR problems. 
To motivate this, it is known that any strongly convex function $f:\mathbb{R}^n \to \mathbb{R}$ satisfies the gradient dominance property; see e.g., \cite[Th. 3.1]{liao2024error}. Indeed, for a strongly convex function $f$
and its optimal value $f^\star$, it is easy to establish the following lower- and upper-bounds for $f-f^\star$:
\begin{subequations} \label{eq:convex-inequalities}
    \begin{align}
        f(x) - f^\star &\geq \frac{\alpha}{2} \|x-x^\star\|^2, \qquad \forall x \in \mathbb{R}^n  \label{eq:QG-strongly-convex}\\
        f(x) - f^\star &\leq \langle \nabla f(x),x - x^\star \rangle, \qquad \forall x \in \mathbb{R}^n \label{eq:gradient-convex}
    \end{align}
\end{subequations}
for some constant $\alpha> 0$, where $x^\star$ is the unique minimizer of $f$. The property \Cref{eq:QG-strongly-convex} is also known as \textit{quadratic growth}, which indicates that the function grows at least as fast as a quadratic. The property \Cref{eq:gradient-convex} is a direct consequence of convexity: $f^\star \geq f(x) + \langle \nabla f(x), x^\star - x\rangle$. We then have 
\begin{align} \label{eq:QG-strongly-convex-1}
f(x) - f^\star \leq  \| \nabla f(x)\|\|x - x^\star\|  
\leq  \| \nabla f(x)\|\sqrt{\frac{2}{\alpha}} \sqrt{ f(x) - f^\star }
\end{align}
where the first inequality follows from the Cauchy-Schwarz inequality for \cref{eq:gradient-convex} and the second inequality applies the quadratic growth \cref{eq:QG-strongly-convex}. We thus have the gradient dominance
$
\frac{\alpha}{2} (f(x) - f^\star) \leq \| \nabla f(x)\|^2, \, \forall x \in \mathbb{R}^n 
$
for any strongly convex function.

Thanks to the hidden convexity, we show that the LQR problems in both continuous-time and discrete-time satisfy two inequalities that slightly generalize \Cref{eq:convex-inequalities}. We have the following key lemma.

\begin{lemma} \label{lemma:LQR-two-key-inequalities}
    Consider the LQR policy optimization problem \cref{eq:LQR_PO_unif}.
Under 
\cref{assumption:stabilizable,assumption:X_positive-definite}, for any stabilizing gain $K \in \mathcal{K}$, 
the optimality gap $J-J^\star$ is lower- and upper-bounded as
\begin{subequations} \label{eq:Jlqr_two-key-inequality}
\begin{align}
J(K)-J^\star  &\geq \left\langle
    R,(K-K^\star)X_K(K-K^\star)^\tr\right\rangle, \label{eq:Jlqr_QG_inprod} \\
    J(K)-J^\star 
    &\leq 
    \langle 
    \nabla J(K) X_K^{-1}X^\star, K-K^\star
    \rangle. \label{eq:Jlqr_weakPL_inprod}
\end{align}
\end{subequations}
Here, $X_K$  denotes the solution to the Lyapunov equation \cref{eq:LQR_unif_Lyapunov} associated with $K$, and $X^\star$ denotes the corresponding Lyapunov solution for the optimal LQR gain $K^\star$.
\end{lemma}

Even though the LQR policy optimization problem \cref{eq:LQR_PO_unif} is nonconvex, a classical change of variables leads to a convex problem that satisfies quadratic growth. We then leverage the convex properties to establish the desired inequalities \cref{eq:Jlqr_two-key-inequality}. We present the convex reformulation in \Cref{subsection:convex_reformulation} and establish the proof of \Cref{lemma:LQR-two-key-inequalities} in \Cref{subsection:proof-lemma-1}. 
We note that the quadratic growth \cref{eq:Jlqr_QG_inprod} is true with only \cref{assumption:stabilizable}; see our discussions after the proof of \Cref{lemma:LQR-two-key-inequalities} in \Cref{remark:QG-LQR}.

With \Cref{lemma:LQR-two-key-inequalities}, we are now ready to establish \cref{theorem:gradient_dominance}, following almost the same steps as the case of strongly convex functions in \cref{eq:QG-strongly-convex-1}. 

\textbf{Proof of \cref{theorem:gradient_dominance}:} The inequality \cref{eq:Jlqr_QG_inprod} implies the following lower bound of $J-J^\star$:
\begin{equation} \label{eq:LQR-QG}
    J(K)-J^\star  \geq \lambda_{\min}(R)\lambda_{\min}(X_K)\|K - K^\star\|_F^2.
\end{equation}
From \Cref{eq:Jlqr_weakPL_inprod}, we have the following upper bound of $J-J^\star$:
\begin{equation} \label{eq:LQR-weakPL}
\begin{aligned}
    J(K)-J^\star 
    &\leq  
    \frac{\lambda_{\max}(X^\star)}{\lambda_{\min}(X_K)} 
    \times
    \|K-K^\star \|_F\|\nabla J(K)\|_F.
\end{aligned}
\end{equation}
Combining \cref{eq:LQR-QG} and \cref{eq:LQR-weakPL}, we get 
\begin{equation} \label{eq:LQR-error-bound}
    \|K-K^\star\|_F
     \leq 
     \frac{\lambda_\mathrm{max}(X^\star)}{\lambda_\mathrm{min}(R)
     \lambda_\mathrm{min}(X_K)^2}
    \|\nabla J(K)\|_F.
\end{equation} 
Substituting \cref{eq:LQR-error-bound} into \cref{eq:LQR-weakPL} yields the inequality \Cref{eq:gradient_dominance_thm} in \cref{theorem:gradient_dominance}. This completes the proof. \hfill $\square$

\begin{remark}[Hidden convexity]\label{remark:hidden_convex_functions}
We say a nonconvex function $f$ has \textit{hidden convexity} if it admits a convex reformulation after a suitable coordinate change. In this case, we can often establish strong global properties of $f$ despite its nonconvexity. This idea of exploiting hidden convexity has been extensively discussed in recent literature   \cite{zheng2023benign,zheng2024benign,fatkhullin2025stochastic,mohammadi2019global,sun2021learning,watanabe2025policy,watanabe2025revisiting,watanabe2025semidefinite}. A simple illustrative case is that $f$ has the composition form $f(x) = g(c(x))$, where  $g:\mathbb{R}^{n}\to\mathbb{R}$ is smooth and convex and $c:\mathbb{R}^n\to\mathbb{R}^n$ is a smooth and invertible mapping.
We have $\nabla f(x) = \nabla c(x)^\tr \nabla g(c(x)) $, where $\nabla c(x)$ is the Jacobian of $c$. 
Moreover, for any global minimizer $x^\star$, we have the bound
\begin{equation} \label{eq:hidden-convex-f}
\begin{aligned}
f(x) - f(x^\star) = g(c(x)) - g(c(x^\star)) &\leq \langle \nabla g(c(x)), c(x) - c(x^\star) \rangle \\
& = \langle (\nabla c(x)^\tr)^{-1} \nabla f(x), c(x) - c(x^\star) \rangle, 
\end{aligned} 
\end{equation}
where the inequality follows from the convexity of $g$ (similar to \cref{eq:gradient-convex}) and the last equality uses the invertibility of the Jacobian $\nabla c(x)$. We see from \cref{eq:hidden-convex-f} that every stationary point of $f$ is globally optimal. 
The inequality \cref{eq:hidden-convex-f} has a strong resemblance to the key inequality \cref{eq:Jlqr_weakPL_inprod} used in the LQR analysis. Indeed, the LQR problem \cref{eq:LQR_unif} admits an analogous hidden convex structure, though with the presence of an additional Lyapunov variable. Our proof of \cref{lemma:LQR-two-key-inequalities} builds on this idea by carefully exploiting such hidden convexity under Lyapunov constraints. As a result, our analysis requires a nonsmooth version of the chain rule used in \cref{eq:hidden-convex-f}. 
\hfill $\square$
\end{remark}

\subsection{A convex reformulation of continuous- and discrete-time LQRs}\label{subsection:convex_reformulation}

It is well known that both continuous-time and discrete-time LQR problems admit equivalent convex reformulations via a change of variables. While the convexification shares great similarities, there is one subtle difference between the continuous-time and discrete-time LQRs: the equivalent form \cref{eq:LQR_unif} can be directly convexified in the continuous-time case, but not in the discrete-time case. We present a simple example to illustrate this point.

\begin{example}[The discrete-time Lyapunov equation cannot be directly convexified]\label{example:lyapunov_eq_convex}
Consider the integrator: $\dot{x}=u$ (in continuous time) and $x_{t+1}=x_t+u_t$ (in discrete-time) with $W=\mathbb{E}[x_0^2]=1$.
We then have $\mathcal{K}_\ct=(-\infty,0)$ and $\mathcal{K}_\dt=(-2,0)$. 
For these systems with $u=kx\,(k\in\mathcal{K})$, the Lyapunov equations \cref{eq:LQR_unif_Lyapunov} are reduced to
$2kX+1=0$ (in CT) and $(k^2-1)X +1=0$ (in DT), respectively.
Now, consider the standard change of variables
$y=kX$ with $X>0$.
For the continuous-time case, the Lyapunov equation is reduced to the linear equation $2y+1 =0$.
However, in discrete-time,
we obtain the nonlinear equation $y^2/X-X+1=0$ with $X>0$, which is a nonconvex equality constraint.
One modification to uniformly convexify both cases is to consider Lyapunov inequalities
$2kX+1\leq 0$ and $(k^2-1)X+1 \leq 0$.
Indeed, the convexity is preserved in the continuous-time case.
In discrete-time, the Lyapunov inequality is reduced to $y^2/X-X+1 \leq0$,
which can be rewritten into
the following (convex) LMI by
the Schur complement since $X>0$:
\begin{equation*}
    y^2/X-X+1 \leq0\quad\Longleftrightarrow \quad
    \begin{bmatrix}
        X-1 & y\\
        y & X
    \end{bmatrix}\succeq 0. 
\end{equation*}
Thus, the 
Lyapunov inequalities can be convexified in both continuous-time and discrete-time.
\hfill$\square$
\end{example}

To unify the convexification process in both continuous-time and discrete-time, we relax the Lyapunov equation \cref{eq:LQR_unif_Lyapunov} into a Lyapunov inequality.
Consider the following problem:
\begin{subequations}\label{eq:LQR_unif_2}
\begin{align}
    \min_{K,\,X}
    \quad& 
    \langle Q+K^\tr R K,X\rangle
    \label{eq:LQR_unif_2-cost}
    \\
    \textrm{subject to}\quad&  \Psi\left(
   \begin{bmatrix}
       A+BK\\
       I
   \end{bmatrix}X
   \begin{bmatrix}
       A+BK\\
       I
   \end{bmatrix}^\tr
   \right)
    +W \preceq 0
    ,\,
    X\succ 0. \label{eq:LQR_unif_2-lyapunov}
\end{align}
\end{subequations}
Two key differences between \cref{eq:LQR_unif} and \cref{eq:LQR_unif_2} are: 1) the Lyapunov equation \cref{eq:LQR_unif_Lyapunov} becomes a Lyapunov inequality \cref{eq:LQR_unif_2-lyapunov} so the feasible region in terms of $X$ is larger in \cref{eq:LQR_unif_2-lyapunov}; 2) the nonconvex stabilizing constraint $K \in \mathcal{K}$ in \cref{eq:LQR_unif_Lyapunov} is replaced by an explicit constraint $X \succ 0$ in \cref{eq:LQR_unif_2-lyapunov}. Despite the two differences, these problems are equivalent in the following sense. A proof is given in \Cref{appendix:proof-LQR-KX-ineq}.

\begin{lemma}\label{proposition:LQR_KX_ineq}
Under \cref{assumption:stabilizable,assumption:X_positive-definite}, problems \cref{eq:LQR_unif,eq:LQR_unif_2} have the same optimal value and the same optimal gain $K^\star$ in both continuous-time and discrete-time settings.
\end{lemma}

From \cref{proposition:LQR_KX_ineq,lemma:LQR-unification}, we know that both continuous- and discrete-time LQR problems \cref{eq:LQR-H2} share the same optimal gain and optimal value with the inequality-constrained problem \cref{eq:LQR_unif_2} under \Cref{assumption:stabilizable,assumption:X_positive-definite}. Furthermore, the optimal gain $K^\star$ to \cref{eq:LQR_unif_2} is also unique by \cref{proposition:uniqueness}. 

We show that the inequality-constrained problem \cref{eq:LQR_unif_2} admits an equivalent convex reformulation. The classical approach is to use a change of variables $Y = KX$ that convexifies the bilinear coupling between $K$ and $X$ in \cref{eq:LQR_unif_2}. To facilitate our analysis, we use a slightly different change of variables: 
\begin{equation}\label{eq:change-of-variable_Y}
    Y=(K-K^\star)X,
\end{equation}
for any feasible $(K,X)$ in \cref{eq:LQR_unif_2}, where $K^\star$ is the unique optimal gain to \cref{eq:LQR_unif_2}. 
Note that from \cref{eq:change-of-variable_Y}, the original gain $K$ can be recovered by
$K = Y X^{-1}+K^\star$ since $X \succ 0$.
Then, we can rewrite \cref{eq:LQR_unif_2} into a convex problem involving the variables $(Y, X)$. In particular, the cost of \cref{eq:LQR_unif_2} becomes  
\begin{align*}
        \langle Q+K^\tr RK,X\rangle
    =
        &
        \langle Q,X \rangle+
        \langle R,(Y+K^\star X) X^{-1}(Y+K^\star X)^\tr\rangle
        \\
        =&
    \langle Q+(K^\star)^\tr RK^\star ,X\rangle
    +
     \langle R,YX^{-1}Y^\tr\rangle
     +2\Tr \left((RK^\star)^\tr Y\right).  
\end{align*}
We can also transform the constraint of \cref{eq:LQR_unif_2}, leading to the following problem:
\begin{subequations}\label{eq:LQR_unif_convex}
\begin{align}
\min_{Y,\,X}
    &\quad 
    \langle Q+(K^\star)^\tr RK^\star ,X\rangle
    +
     \langle R,YX^{-1}Y^\tr\rangle
     +2\Tr \left((RK^\star)^\tr Y\right)
     \label{eq:LQR_unif_convex-cost}
    \\
    \text{subject to}&\quad
    \Psi
   \left(
   \begin{bmatrix}
       (A+BK^\star)X+BY\\
       X
   \end{bmatrix}X^{-1}
   \begin{bmatrix}
       (A+BK^\star)X+BY\\
       X
   \end{bmatrix}^\tr
   \right)
    +W \preceq 0,\,
    X\succ 0.
    \label{eq:LQR_unif_convex_Lyapunov}
\end{align}
\end{subequations}

From \cref{eq:change-of-variable_Y}, let us define a mapping $\Upsilon:\mathbb{R}^{m\times n}\times\mathbb{S}_{++}^n\to \mathbb{R}^{m\times n}\times\mathbb{S}_{++}^n$ as 
    \begin{equation} \label{eq:change-of-variable-X-Y}
        \Upsilon(K,X) = \left((K-K^\star)X,X\right).
    \end{equation}
This mapping is infinitely differentiable, and so is its inverse $\Upsilon^{-1}(Y,X)=(YX^{-1}+K^\star,X)$. It is clear that problems \Cref{eq:LQR_unif_convex,eq:LQR_unif_2} are equivalent under this mapping $\Upsilon$.

\begin{lemma} \label{lemma:convex-reformulation}
Suppose \cref{assumption:stabilizable,assumption:X_positive-definite} hold.
The following statements hold for both continuous- and discrete-time settings.
\begin{enumerate}
    \item Problems \Cref{eq:LQR_unif_convex,eq:LQR_unif_2} are equivalent in the sense that the feasible regions of \cref{eq:LQR_unif_2} and \cref{eq:LQR_unif_convex} are related by the invertible mapping $\Upsilon$, i.e.,  
        \begin{equation*}
        \left\{
        (Y,X) 
        \middle|
        \text{ \cref{eq:LQR_unif_convex_Lyapunov} holds}
        \right\} = \Upsilon\left(
        \left\{
        (K,X)
        \middle|
        \text{ \cref{eq:LQR_unif_2-lyapunov} holds}
        \right\}
        \right).
        \end{equation*}
    Moreover, they have the same cost for any feasible points satisfying $(Y,X)=\Upsilon(K,X)$. 
    
    \item 
    Problem \cref{eq:LQR_unif_convex}
    is convex, i.e.,
    both
    the cost \cref{eq:LQR_unif_convex-cost} and constraints \cref{eq:LQR_unif_convex_Lyapunov} are convex. 
    Moreover,
    an optimal solution to \cref{eq:LQR_unif_convex} is given by $(0,X^\star)$, where $(K^\star,X^\star)$ is an optimal solution to \cref{eq:LQR_unif}. 
\end{enumerate}
\end{lemma}

The first statement is straightforward from the construction of \cref{eq:LQR_unif_convex}. 
For the second statement, 
it is well-known that problem \cref{eq:LQR_unif_convex} is convex.
We give some details to illustrate the convexity of \cref{eq:LQR_unif_convex} in discrete-time
in \cref{subsec-proof:convex_derivation},
where we use the Schur complement as \cref{example:lyapunov_eq_convex}.
For the continuous-time case, it is easier to see the convexity of \cref{eq:LQR_unif_convex}, since
the constraint \cref{eq:LQR_unif_convex_Lyapunov} with $\Psi=\Psi_\ct$ is an LMI.
The optimality of $(0,X^\star)$ for \cref{eq:LQR_unif_convex} is clear from the construction $\Upsilon(K^\star,X^\star)$.

We remark that the definition of $Y$ in \cref{eq:change-of-variable_Y} yields a simple optimal solution (the optimality of \cref{eq:LQR_unif_convex} with respect to $Y$ is achieved by $Y^\star=0$), which is more convenient for our subsequent analysis.

\subsection{Proof of \Cref{lemma:LQR-two-key-inequalities}} \label{subsection:proof-lemma-1}

We here prove the key inequalities in  
\cref{lemma:LQR-two-key-inequalities}. 
Our proof utilizes the convex reformulation \cref{eq:LQR_unif_convex}
and applies uniformly to both continuous-time and discrete-time settings. 

\begin{subequations}
To facilitate our discussions, let us define a convex function $f_\mathrm{cvx}(Y,X):\mathbb{R}^{m\times n}\times  \mathbb{S}_{++}^n\to\mathbb{R}$ as 
\begin{equation} \label{eq:convex-LQR-cost}
f_\mathrm{cvx}(Y,X) = \langle Q+(K^\star)^\tr RK^\star ,X\rangle
    +
     \langle R,YX^{-1}Y^\tr\rangle
     +2\Tr \left((RK^\star)^\tr Y\right)
\end{equation}
and a convex set
\begin{equation}
    \mathcal{F}_\mathrm{cvx}=\left\{(Y,X)
    \times \mathbb{R}^{m\times n}\times \mathbb{S}_{++}^n
    \middle|
    \text{ \cref{eq:LQR_unif_convex_Lyapunov} holds}
    \right\}.
\end{equation}
\end{subequations}
Then, we can represent \cref{eq:LQR_unif_convex} as $\min \, f_\mathrm{cvx}(Y,X)\; \text{subject to}\, (Y,X) \in \mathcal{F}_\mathrm{cvx}$.

\subsubsection{The lower bound \cref{eq:Jlqr_QG_inprod}}\label{subsubsection:proof_QG}

Here, we show a further property beyond convexity in problem  \cref{eq:LQR_unif_convex}. In particular, the cost $f_\mathrm{cvx}$ in \cref{eq:convex-LQR-cost} enjoys a partial quadratic growth property in $Y$. 

\begin{proposition}[Partial quadratic growth]\label{lemma:qg_Jcvx}
Suppose \cref{assumption:stabilizable,assumption:X_positive-definite} hold.
Then, we have
\begin{align}\label{eq:partial_qg}
      f_\mathrm{cvx} (Y,X)-f_\mathrm{cvx}^\star
      \geq
       \langle
    R,YX^{-1}Y^\tr
    \rangle, \quad \forall (Y,X) \in \mathcal{F}_\mathrm{cvx},
\end{align}
where $f_\mathrm{cvx}^\star$ is the optimal value of \cref{eq:LQR_unif_convex}. 
\end{proposition}
\begin{proof}
 Let $(K^\star, X^\star)$ be an optimal solution to \cref{eq:LQR_unif_2}. For notational simplicity, we denote $Q^\star = Q+(K^\star)^\tr RK^\star$. By \cref{proposition:LQR_KX_ineq,lemma:convex-reformulation}, we know that the optimal LQR value is 
$
f^\star_{\mathrm{cvx}}=J^\star = \langle Q+(K^\star)^\tr RK^\star,X^\star \rangle=\langle Q^\star,X^\star \rangle. 
$
By direct computation, we have 
\begin{align}\label{eq:fcvx_gap_proof}
    f_\mathrm{cvx}(Y,X)-f^\star_{\mathrm{cvx}} 
    = 
    \langle Q^\star,X-X^\star \rangle
     +2\langle RK^\star, Y\rangle
     + \langle R,YX^{-1}Y^\tr\rangle, \quad \forall (Y,X) \in \mathcal{F}_\mathrm{cvx}. 
\end{align}

Now, using standard matrix calculus \cite{petersen2008matrix}, we can compute the gradient 
of $\nabla f_\mathrm{cvx}(Y,X)$ at $(0,X^\star)$ as
$\nabla f_{\mathrm{cvx}}\left(0, X^{\star}\right) = \begin{bmatrix}
    2 R K^{\star} \\ Q^\star
\end{bmatrix}$
.
The optimality of $(0,X^\star)$ to the convex problem \cref{eq:LQR_unif_convex} (see e.g., \cite[Th. 6.12]{rockafellar1998variational}) implies that 
\begin{align*}
     \left\langle \nabla f_{\mathrm{cvx}}\left(0, X^{\star}\right),\begin{bmatrix}
Y-0 \\
X-X^{\star}
\end{bmatrix} \right\rangle =  \langle Q^\star,X-X^\star \rangle
     +2\langle RK^\star, Y\rangle \geq 0, \quad \forall (Y,X) \in \mathcal{F}_\mathrm{cvx}.
\end{align*}
We complete the proof by substituting this into \cref{eq:fcvx_gap_proof}. 
\end{proof}

With \Cref{lemma:qg_Jcvx}, it is now easy to derive the lower bound \cref{eq:Jlqr_QG_inprod}.

\vspace{1mm}

\noindent \textbf{Proof of \cref{eq:Jlqr_QG_inprod}}: For any stabilizing gain $K \in \mathcal{K}$, we compute its associated Lyapunov variable $X_K$ to \cref{eq:LQR_unif_Lyapunov}. We know $X_K \succ 0$ by \Cref{assumption:X_positive-definite}. Let 
$(Y,X)=\Upsilon(K,X_K)$ 
which by construction satisfies $(Y,X) \in \mathcal{F}_{\mathrm{cvx}}$ and $f_\mathrm{cvx}(Y,X)=J(K)$. Substituting $(Y,X)=\Upsilon(K,X_K)$ into \cref{eq:partial_qg},
we immediately obtain 
$$
J(K) - J^\star = f_\mathrm{cvx} (Y,X)-f_\mathrm{cvx}^\star \geq \langle
    R,YX^{-1}Y^\tr
    \rangle = \left\langle
    R,(K-K^\star)X_K(K-K^\star)^\tr\right\rangle, \;\; \forall K \in \mathcal{K}.
$$

\begin{remark}[Quadratic growth] \label{remark:QG-LQR}
For deriving 
the lower-bound of $J-J^\star$ in \cref{eq:Jlqr_QG_inprod},
\cref{assumption:X_positive-definite} is not necessarily required.
In fact, 
we can show \cref{eq:Jlqr_QG_inprod} by the \textit{completion of squares} technique
for both continuous-time and discrete-time systems.
For example, in the discrete-time case,
we know \begin{align*}
    J(K)
    = &\langle P^\star ,W\rangle
    + 
    \mathbb{E}\left[
     \sum_{t=0}^\infty
    (u_t-K^\star x_t)^\tr
    (R+B^\tr P^\star B)
    (u_t-K^\star x_t) 
    \right]
    \\
    =&\langle P^\star ,W\rangle
    +
    \mathbb{E}\left[
     \sum_{t=0}^\infty
    x_t^\tr (K-K^\star )^\tr
    (R+B^\tr P^\star B)(K-K^\star)x_t
    \right]
    ,
\end{align*}
which only requires \cref{assumption:stabilizable}
\cite[Lem 16.6.3]{lancaster1995algebraic}
and yields \cref{eq:Jlqr_QG_inprod} from \cref{eq:second-order-moment-state}.
A similar result holds in continuous time (see e.g., \cite[Lem 16.3.2]{lancaster1995algebraic}).
Our proof presents an alternative way 
through a convex reformulation,
which also allows us to show the upper-bound \cref{eq:Jlqr_weakPL_inprod}.
\hfill$\square$
\end{remark}

\subsubsection{The upper bound \cref{eq:Jlqr_weakPL_inprod}}\label{subsubsection:proof_weakPL}

Our key idea in establishing \cref{eq:Jlqr_weakPL_inprod} is to utilize the convex property \cref{eq:gradient-convex} and the chain rule for computing (sub)gradients. For convenience, we introduce the unconstrained form of \cref{eq:LQR_unif_convex}.
   Let 
   \begin{align} \label{eq:LQR-convex-unconstrained}
    \tilde{f}_\mathrm{cvx} (Y,X):=
        f_\mathrm{cvx}(Y,X)
        + \delta_{\mathcal{F}_\mathrm{cvx}}(Y,X), \quad \forall (Y,X) \in \mathbb{R}^{m \times n} \times 
        \mathbb{R}^{n\times n}, 
    \end{align}
    where 
    $\delta_{\mathcal{F}_\mathrm{cvx}}(\cdot)$ denotes the indicator function of $\mathcal{F}_\mathrm{cvx}$, i.e.,
    $\delta_{\mathcal{F}_\mathrm{cvx}}(Y,X)=0$ for $(Y,X)\in\mathcal{F}_\mathrm{cvx}$ and
     $\delta_{\mathcal{F}_\mathrm{cvx}}(Y,X)=\infty$ for $(Y,X)\notin\mathcal{F}_\mathrm{cvx}$.
    Then, \cref{eq:LQR_unif_convex} is equivalent to  
    $
    \min_{Y,X}\,\tilde{f}_\mathrm{cvx}(Y,X).
    $

    Since $\tilde{f}_\mathrm{cvx}$ is nonsmooth but convex, we can define its usual convex subdifferential as 
    $$
    \partial \tilde{f}_\mathrm{cvx}(Y,X) = \{(H_1,H_2) 
    \mid \tilde{f}_\mathrm{cvx}(\hat{Y},\hat{X}) \geq \tilde{f}_\mathrm{cvx}(Y,X) + \langle H_1, \hat{Y} - Y  \rangle + \langle H_2, \hat{X} - X
    \rangle,\,
    \forall (\hat{Y},\hat{X})
    \}
    $$
    at any point $(Y,X) \in \mathcal{F}_{\mathrm{cvx}}$. Thanks to \cref{proposition:LQR_KX_ineq} and the change of variables in \cref{eq:change-of-variable-X-Y} that connects \Cref{eq:LQR_unif_convex,eq:LQR_unif_2}, we can explicitly construct a subgradient in $\partial \tilde{f}_\mathrm{cvx}(Y_K,X_K)$ using the gradient $\nabla J(K)$. In particular, we have the following key result. %

    \begin{proposition}\label{proposition:fcvx_subgradient_J}
        Suppose \cref{assumption:stabilizable,assumption:X_positive-definite} hold. Let $K\in\mathcal{K}$, and denote its associated Lyapunov variable to \cref{eq:LQR_unif_Lyapunov} as $X_K$. Define $Y_K=(K-K^\star) X_K$. Then, we have
    \begin{equation}\label{eq:gradJ_fcvx}
     (   \nabla J(K)X_K^{-1},
         (K^\star-K)^\tr
    \nabla J(K)X_K^{-1})
     \in
       \partial \tilde{f}_\mathrm{cvx}
       \left(Y_K,X_K\right).
    \end{equation}
    \end{proposition}
   By construction, we have  $(Y_K,X_K) \in \mathcal{F}_\mathrm{cvx}$, so that $\partial \tilde{f}_\mathrm{cvx}
       \left(Y_K,X_K\right)$ is not empty. \Cref{proposition:fcvx_subgradient_J} explicitly constructs one subgradient in $\partial \tilde{f}_\mathrm{cvx}
       \left(Y_K,X_K\right)$ using the gradient $\nabla J(K)$. 
    To establish this result, we use the chain rule for nonsmooth functions \cite{rockafellar1998variational} since $\tilde{f}_\mathrm{cvx}$ is nonsmooth, which involves
    a calculation of the Jacobian of $\Upsilon^{-1}$ for the invertible mapping $\Upsilon$ in
    \cref{eq:change-of-variable-X-Y}.
We present the proof details in \cref{section:proof_partialmin_chain-rule}.\footnote{We also present an alternative proof without using the chain rule in \cref{subsection:proof-2-fcvx_subgradient}.
}
With \Cref{proposition:fcvx_subgradient_J}, we are ready to establish the upper bound \cref{eq:Jlqr_weakPL_inprod}.

\vspace{1mm}

    \textbf{Proof of \cref{eq:Jlqr_weakPL_inprod}.}     Let $(K^\star, X^\star)$ be the optimal solution to \cref{eq:LQR_unif}, with optimal value as $J^\star$. By \Cref{lemma:convex-reformulation}, we know that $(0,X^\star)$ is a minimizer of $\tilde{f}_{\mathrm{cvx}}$. Then, similar to the inequality \cref{eq:gradient-convex},    
   by the definition of convex subdifferential, we have 
   \begin{equation}\label{eq:fcvx_convexity_proof}
        \tilde{f}_\mathrm{cvx}(Y,X)
        -f^\star_{\mathrm{cvx}}
        \leq 
        \langle 
            H_1, Y - 0 \rangle + \langle H_2,
            X-X^\star
        \rangle,\quad
      \forall  (H_1,H_2)\in \partial \tilde{f}_\mathrm{cvx}(Y,X). 
    \end{equation}
For any stabilizing gain $K\in\mathcal{K}$ and its associated Lyapunov variable $X_K \succ 0$ to \cref{eq:LQR_unif_Lyapunov}, we have by construction that 
     $   (Y_K,X_K)=\left((K-K^\star)X_K,X_K\right)\in \mathcal{F}_\mathrm{cvx}, $
    $J(K) = \tilde{f}_\mathrm{cvx}(Y_K,X_K)$,  and $J^\star = f_{\mathrm{cvx}}^\star$. 
    
    We apply the inequality \cref{eq:fcvx_convexity_proof} to this pair $(Y_K,X_K)$, leading to 
    \begin{equation} \label{eq:fcvx-Jk}
    \begin{aligned}
    J(K)-J^\star
         =
        \tilde{f}_\mathrm{cvx}(Y_K,X_K)-J^\star
\leq &
        \langle 
        H_1, Y_K \rangle + \langle H_2, X_K-X^\star
        \rangle,
    \end{aligned}
    \end{equation}
    where $(H_1,H_2) \in \partial \tilde{f}_\mathrm{cvx}(Y_K,X_K)$. 
Now substituting \cref{eq:gradJ_fcvx} into \cref{eq:fcvx-Jk} leads to 
$$
\begin{aligned}
    J(K)-J^\star &\leq \langle  \nabla J(K)X_K^{-1}, Y_K \rangle  - \langle  (K-K^\star)^\tr
    \nabla J(K)X_K^{-1}, X_K - X^\star\rangle \\
    &=\langle  \nabla J(K)X_K^{-1}, (K-K^\star)X_K \rangle  - \langle  (K-K^\star)^\tr
    \nabla J(K)X_K^{-1}, X_K - X^\star\rangle \\
    &= \langle  
    \nabla J(K)X_K^{-1}, (K-K^\star)X^\star\rangle \\
    &= \langle \nabla J(K) X_K^{-1}X^\star,K-K^\star \rangle,
\end{aligned}
$$
which confirms the desired upper bound \cref{eq:Jlqr_weakPL_inprod}.
\hfill$\square$

\section{Proof of \Cref{proposition:fcvx_subgradient_J} via the chain rule}\label{section:proof_partialmin_chain-rule}

Our proof of \Cref{proposition:fcvx_subgradient_J} leverages the connections among the original LQR \cref{eq:LQR_PO_unif}, its equivalent inequality form \cref{eq:LQR_unif_2}, and the convex reformulation \cref{eq:LQR_unif_convex}. Both \cref{eq:LQR_unif_2} and \cref{eq:LQR_unif_convex} have an additional Lyapunov variable $X$, which involves explicit linear matrix inequality constraints. To facilitate our discussion, we define the unconstrained form of \cref{eq:LQR_unif_2}
    using the indicator function as 
         \begin{equation}
         J_\mathrm{lft}(K,X):=
    \langle 
    Q+K^\tr RK,X
    \rangle
    + \delta_{\mathcal{L}_\mathrm{lft}}(K,X), 
         \end{equation}
    where 
     $   \mathcal{L}_\mathrm{lft}
        := \left\{
        (K,X)\in \mathbb{R}^{m\times n}\times \mathbb{S}^n
        \middle| \text{ \cref{eq:LQR_unif_2-lyapunov} holds} 
        \right\}.$ 
    From the proof of \cref{proposition:LQR_KX_ineq}, for any stabilizing gain $K \in \mathcal{K}$ and its associated Lyapunov solution $X_K$, we have the partial minimization property:
        \begin{equation} \label{eq:connection-J-Jlft}
        J(K) = \min_{X\in \mathbb{S}^n} J_\mathrm{lft}(K,X), \qquad X_K\in\arg\min_{X\in \mathbb{S}^n} J_\mathrm{lft}(K,X).
        \end{equation}     
    From the convex reformulation in \cref{subsection:convex_reformulation}, 
    we have 
     $   {J}_\mathrm{lft}(K,X)
        = (\tilde{f}_\mathrm{cvx}\circ \Upsilon)(K,X)
    $
    whenever $X \in \mathbb{S}^n_{++}$,  where $\tilde{f}_\mathrm{cvx}$ is given in \cref{eq:LQR-convex-unconstrained} and $\Upsilon$ is defined in \cref{eq:change-of-variable-X-Y}. For the invertible mapping $\Upsilon$, let us denote $\Pi=\Upsilon^{-1}$, given by
    \begin{equation} \label{eq:inverse-mapping}
    \Pi(Y,X)=(YX^{-1}+K^\star,X), \qquad \forall (Y,X) \in \mathbb{R}^{m \times n} \times \mathbb{S}^n_{++}. 
    \end{equation}
    Then, it holds for any $X\in \mathbb{S}_{++}^n$
    that
    \begin{equation} \label{eq:connectin-Jlft-fcvx}
\tilde{f}_\mathrm{cvx}(Y,X) = ({J}_\mathrm{lft}\circ\Pi)(Y,X).
    \end{equation}

    We will prove \Cref{proposition:fcvx_subgradient_J} based on \cref{eq:connection-J-Jlft} and \cref{eq:connectin-Jlft-fcvx}. 
    Since ${J}_\mathrm{lft}$ is nonsmooth and nonconvex, we here introduce 
   an extension of subgradients to nonconvex functions, i.e., Fr\'{e}chet subgradients. 
    \begin{definition}[Fr\'{e}chet subgradient\footnote{The definition can be naturally extended to matrix-valued functions.}]\label{definition:Frechet_subgradient}
    For a function $f:\mathbb{R}^n\to\mathbb{R}\cup\{+\infty\}$
    and a point $x$ at which $f(x)$ is finite, vector $v$ is said to be
    a \textit{Fr\'{e}chet subgradient}
    at $x$ if
    \begin{equation*}
       \underset{y\to x}{\liminf}
        \frac{f(y)-f(x)- \langle v,y-x \rangle}{\|y-x\|} \geq 0.
    \end{equation*}  
    \end{definition}
Let $\hat \partial f(x)$ denote the Fr\'{e}chet subdifferential at $x$, i.e., the set of all Fr\'{e}chet subgradients at $x$. When $f$ is convex, $\hat \partial f(x)$ becomes the usual convex subdifferential. If $f$ is smooth, then we have $\hat \partial f(x) = \{\nabla f(x)\}$, i.e., the Fr\'{e}chet subdifferential becomes the singleton with $\nabla f(x)$. 

Thanks to \cref{eq:connection-J-Jlft}, we can connect the gradient $\nabla J(K)$ to a Fr\'{e}chet subgradient of $\hat{\partial} J_\mathrm{lft}(K,X_K)$. 
    \begin{lemma}\label{proposition:subgradient_Jlft}
    With the same setting in \Cref{proposition:fcvx_subgradient_J}, 
    it holds
    for $K\in\mathcal{K}$ that
\begin{equation}\label{eq:partial_minimization_grad}
            (\nabla J(K),
            0)
        \in 
       \hat{\partial} J_\mathrm{lft}(K,X_K).
    \end{equation}
    \end{lemma}
    \begin{proof} 
    We verify this lemma directly using \Cref{definition:Frechet_subgradient}.
    Since $J$ is differentiable over $\mathcal{K}$, we have 
    \begin{equation*}
    J(\hat{K}) = J(K) + \langle \nabla J(K),\hat K-K \rangle + o(\|\hat K-K\|_F).
    \end{equation*}
    From \cref{eq:connection-J-Jlft}, we see that
    $J_\mathrm{lft}(\hat{K},\hat X)\geq J(\hat{K}),\,\forall \hat X\in\mathbb{R}^{n\times n}$
    and
    $J(K)=J_\mathrm{lft}(K,X_K)$.
    Then,
    \begin{equation}\label{eq:Frechet_Jlft}
    \begin{aligned}
        J_\mathrm{lft}(\hat{K},\hat X)
        \geq J(\hat{K}) &= J(K) + \langle \nabla J(K),\hat K-K \rangle + o(\|\hat K-K\|_F) \\
        &= J_\mathrm{lft}(K,X_K)
        + \langle 
        \nabla J(K),\hat K-K
        \rangle
        +
        \langle 
        0,\hat X-X_K
        \rangle
        +o(\|\hat K-K\|_F). 
    \end{aligned}
    \end{equation}
    Thus, we have
    $$\liminf_{\hat K\to K}
    \frac{
    J_\mathrm{lft}(\hat{K},\hat X)-J_\mathrm{lft}(K,X_K)
    -\langle 
        \nabla J(K),\hat K-K
        \rangle
        -
        \langle 
        0,\hat X-X_K
        \rangle}{\|(\hat K,\hat X)-(K,X_K)\|_F}\geq0.
    $$ 
    By the definition of the Fr\'{e}chet subgradient, we confirm \cref{eq:partial_minimization_grad}.
    \end{proof}

{
\cref{proposition:subgradient_Jlft} 
is similar to \textit{Danskin's theorem} \cite{bertsekas1997nonlinear} and
may seem obvious from 
the partial minimization property
$J(K)=\min_X J_\mathrm{lft}(K,X)=J_\mathrm{lft}(X,X_K)$.
However, in our case of $J_\mathrm{lft}$, it may not simply follow from 
standard results 
(e.g., \cite[Sec.~10.C]{rockafellar1998variational}),
due to 
the indicator function $\delta_{\mathcal{L}_\mathrm{lft}}$ and nonconvexity. 
We have used a direct proof based on \Cref{definition:Frechet_subgradient}
using the differentiability of $J$,
which is in general not guaranteed under partial minimization.
}

Thanks to 
$\tilde f_\mathrm{cvx}=J_\mathrm{lft}\circ \Pi$ in \cref{eq:connectin-Jlft-fcvx}, we next introduce the chain rule to compute a subgradient in $\partial \tilde{f}_{\mathrm{cvx}}$ using a Fr\'{e}chet subgradient in $\hat{\partial} J_\mathrm{lft}$. In particular, the standard chain rule
    for nonsmooth functions \cite[Th. 10.6]{rockafellar1998variational}\footnote{
    The term ``regular subgradient'' is used for ``Fr\'{e}chet subgradient''; see \cite[Def. 8.3 (a)]{rockafellar1998variational}. In \cref{eq:chain-rule_vec}, we have used suitable vectorization to handle the matrix variables; also $J_\mathrm{lft}$ is lower semi-continuous at $\Pi
        (Y,X)$ for any  $(Y,X) \in \mathcal{F}_\mathrm{cvx}$ (see \cref{subsubsec-proof:chain-rule}
     for further details). 
    }
    guarantees that for any $(Y,X) \in \mathcal{F}_\mathrm{cvx}$
    and
        $ 
    (G_1,G_2) \in\hat{\partial} J_\mathrm{lft}
        (\Pi
        (Y,X) )$, we have 
     \begin{equation}\label{eq:chain-rule_vec}
    (H_1,
    H_2)
       \in
       \partial \tilde{f}_\mathrm{cvx}(Y,X)
       \quad\text{with}\quad
               \begin{bmatrix}
           \operatorname{vec} (H_1)\\
            \operatorname{vec} (H_2)
        \end{bmatrix}
        =\Gamma
        (Y,X)^\tr 
        \begin{bmatrix}
            \mathrm{vec}(G_1)\\
            \mathrm{vec}(G_2)
        \end{bmatrix},
    \end{equation}
    where $\Gamma (Y,X)$ corresponds to the Jacobian of $\Pi(Y,X)=(YX^{-1}+K^\star,X)$ with vectorization:
        \begin{equation} \label{eq:Jacobian-Pi}
            \Gamma(Y,X):=
    \begin{bmatrix}
        \frac{\partial }{\partial \operatorname{vec}^\tr(Y) }
         \operatorname{vec}(YX^{-1}+K^\star)
        &
         \frac{\partial }{\partial \operatorname{vec}^\tr(X) }
          \operatorname{vec}(YX^{-1}+K^\star)
         \\
          \frac{\partial }{\partial \operatorname{vec}^\tr(Y) }
          \operatorname{vec}(X)
          &
         \frac{\partial }{\partial \operatorname{vec}^\tr(X) }
         \operatorname{vec}(X)
    \end{bmatrix}.
        \end{equation}
    Here, we use $\operatorname{vec}(\cdot)$ to denote the standard vectorization of a matrix. 
    We now present the explicit form of Jacobian $\Gamma(Y,X)$ in \cref{eq:chain-rule_vec}.
    See \cref{subsubsec-proof-Jacobian} for a proof.
    \begin{lemma}[Explicit form of the Jacobian]\label{proposition:Jacobian}
    Define the vectorized form of the mapping $\Pi$ as 
    \begin{equation*}
         \hat{\Pi}\left(
    \begin{bmatrix}
        \operatorname{vec}(Y)\\\operatorname{vec}(X)
    \end{bmatrix}
    \right)
    = \begin{bmatrix}
        \operatorname{vec}(YX^{-1}+K^\star)\\
        \operatorname{vec}(X)
    \end{bmatrix},\qquad
    (Y,X)\in \mathbb{R}^{m\times n}\times
    \mathbb{S}^{n\times n}_{++}.
    \end{equation*}
    Then, the Jacobian $
    \nabla \tilde{\Pi}$ can be computed explicitly as 
    \begin{align}\label{eq:Gamma_Jacobian}
    \Gamma(Y,X)
    =\nabla \hat{\Pi}\left(
    \begin{bmatrix}
        \operatorname{vec}(Y)\\\operatorname{vec}(X)
    \end{bmatrix}
    \right)
    = \begin{bmatrix}
    X^{-1}\otimes I_m
    &
    -(X^{-1}\otimes YX^{-1})\\
    0& I_{n^2}
    \end{bmatrix} .
\end{align}
    \end{lemma} 

With \Cref{proposition:Jacobian,proposition:subgradient_Jlft} and the chain rule \cref{eq:chain-rule_vec}, the proof of \Cref{proposition:fcvx_subgradient_J} becomes immediate.

\vspace{1mm}

\noindent \textbf{Proof of \Cref{proposition:fcvx_subgradient_J}}. Let $K\in\mathcal{K}$, and denote its associated Lyapunov variable to \cref{eq:LQR_unif_Lyapunov} as $X_K$. Define $Y_K=(K-K^\star) X_K$. By definition, we have $(K,X_K) = \Pi(Y_K,X_K)$. 

From \cref{proposition:subgradient_Jlft}, we know that $(\nabla J(K),
            0)
        \in 
       \hat{\partial} J_\mathrm{lft}(K,X_K)$.  
    The chain rule \cref{eq:chain-rule_vec} implies that we have $(H_1,H_2)\in\partial \tilde{f}_\mathrm{cvx}(Y_K,X_K)$, where 
    \begin{align*}
         \begin{bmatrix}
             \operatorname{vec}(H_1)\\
             \operatorname{vec}(H_2)
         \end{bmatrix}
         \overset{\cref{eq:partial_minimization_grad}}{= }&
        \Gamma (Y_K,X_K)^\tr \begin{bmatrix}
            \operatorname{vec}(\nabla J(K))\\
            0 
        \end{bmatrix} \\
         \overset{\cref{eq:Gamma_Jacobian}}{= }&
         \begin{bmatrix}
        X_K^{-1}\otimes I_m & 0    \\
        -(X_K^{-1}\otimes (K-K^\star)^\tr)  & I_{n^2}
        \end{bmatrix}
        \begin{bmatrix}
            \operatorname{vec}(\nabla J(K))\\
            0 
        \end{bmatrix}  \\
        =&
        \begin{bmatrix}
            \operatorname{vec} (   \nabla J(K)X_K^{-1})\\
         \operatorname{vec}
         \left((K^\star-K)^\tr
    \nabla J(K)X_K^{-1}\right)
        \end{bmatrix}.
    \end{align*}
    Here, the final equality follows from the fact that 
$(S_2^\tr\otimes S_1)\operatorname{vec}(V)
=\operatorname{vec}(S_1VS_2)$ \cite[Prop. 5.2.1]{lancaster1995algebraic}. 
This implies  $H_1=\nabla J(K)X_K^{-1}$ and $
    H_2=(K^\star-K)^\tr
    \nabla J(K)X_K^{-1}$, and thus \cref{eq:gradJ_fcvx} holds.
    \hfill$\square$

%% file: sections/appendix-technical-details.tex
\section{Technical backgrounds in \cref{section:problem_formulation,section:gradient_dominance}}\label{appendix:backgrounds_sec2-3}

This section reviews some technical backgrounds for our results in \cref{section:problem_formulation,section:gradient_dominance}.
\subsection{Lyapunov equations}\label{subsection:Lyapunov_summary}

Here, we summarize the well-known and fundamental properties of Lyapunov equations.
Consider the Lyapunov equations for $A\in\mathbb{R}^{n\times n}$ and $W\in\mathbb{S}_{+}^n$:
\begin{align}\label{eq:Lyapunov_general}
\begin{cases} 
    AX + XA^\tr +W =0 & \text{(CT)}\\
    AXA^\tr - X +W =0, & \text{(DT)}
\end{cases}
\end{align}
where CT and DT stand for continuous-time and discrete-time, respectively.

For \cref{eq:Lyapunov_general}, the following properties are widely recognized in the literature.
See, e.g., \cite[Sec. 5]{lancaster1995algebraic} and \cite[Sec. 3.8]{zhou1996robust}.
We frequently use these properties in our analysis.
\begin{proposition}[Fundamental properties of Lyapunov equations]\label{proposition:Lyapunov_general}
Consider the Lyapunov equations in \cref{eq:Lyapunov_general}.
Assume that $A$ is Hurwitz stable in CT (resp. Schur stable in DT)
and  $W$ is positive semidefinite.
Then, the following statements hold:
\begin{enumerate}
    \item The solution $X$ is unique and given by the following positive semidefinite matrix:
    \begin{multicols}{2}
\setlength{\columnseprule}{0.4pt}
    \begin{equation*}
        X = 
            \displaystyle \int_0^\infty e^{At}We^{A^\tr t}dt\quad\text{(CT)}
    \end{equation*}
\columnbreak
    \begin{equation*}
        X = 
            \sum_{t=0}^\infty A^t W (A^\tr)^t.
            \quad \text{(DT)}
    \end{equation*}
\end{multicols}
\vspace{-4mm}
\noindent Moreover, $X$ admits the following alternative representation after the vectorization\footnote{
    This representation clearly shows that $X$ depends on $A$ continuously. This fact gives rise to \cref{Corollary:gradient-dominance-compact} as a direct corollary of \cref{theorem:gradient_dominance}.
    In addition, the continuous differentiability of $J$ also becomes clear from this vectorized form.
    }:
        \begin{multicols}{2}
\setlength{\columnseprule}{0.4pt}
    \begin{equation*}
        \operatorname{vec}(X) =
            -(I\otimes A+A\otimes I)^{-1}\operatorname{vec}(W)
            \quad\text{(CT)}
    \end{equation*}
\columnbreak
    \begin{equation*}
    \operatorname{vec}(X) =
        (I-A\otimes A)^{-1}\operatorname{vec}(W).
            \quad\text{(DT)}
    \end{equation*}
\end{multicols}
    \vspace{-4mm}
    \item The solution $X$ is positive definite if and only if $(A,W^{1/2})$ is controllable.\footnote{
    If $W\succ0$, $(A,W^{1/2})$ is always controllable.}
    \item Consider the Lyapunov equations \cref{eq:Lyapunov_general} and 
    \begin{align}\label{eq:Lyapunov_general_2}
\begin{cases} 
    A\hat X + \hat XA^\tr +\hat W =0 & \text{(CT)}\\
    A\hat XA^\tr -\hat  X + \hat W =0. & \text{(DT)}
    \end{cases}
\end{align}
    If $\hat W\succeq  W$ holds, then we have $\hat X\succeq X$
    in both continuous-time and discrete-time cases.
\end{enumerate}
\end{proposition}

The first part in \Cref{proposition:Lyapunov_general} presents the explicit forms of the unique solution $X$.
It is worth noting that the first statement  can also be viewed as
the following cumulative covariance
    \begin{multicols}{2}
\setlength{\columnseprule}{0.4pt}
    \begin{equation*}
            X = 
            \displaystyle \int_0^\infty 
            \mathbb{E}[x(t)x(t)^\tr ]
            \qquad\text{(CT)}
    \end{equation*}
\columnbreak
    \begin{equation*}
        X = 
                     \sum_{t=0}^\infty \mathbb{E}[x_tx_t^\tr ] \qquad \text{(DT)}
    \end{equation*}
\end{multicols}
\vspace{-4mm}
\noindent for the linear system
\begin{equation*}
    \begin{cases}
    \dot x(t) = Ax(t) &\text{(CT)}\\
        x_{t+1} = Ax_t, &\text{(DT)}
    \end{cases}\qquad\text{with}\qquad
    \mathbb{E}[x_0]=0,\quad \mathbb{E}[x_0x_0^\tr ]=W.
\end{equation*}
The second part characterizes positive definite solutions to \cref{eq:Lyapunov_general}. 
We use this to analyze \cref{assumption:X_positive-definite} for \cref{theorem:gradient_dominance} (see \cref{proposition:X_positive}).
Finally, the third part shows a monotonicity in Lyapunov equations.
This property is a key to developing a convex reformulation in \cref{subsection:convex_reformulation}, in particular \cref{proposition:LQR_KX_ineq}.

\subsection{Proof of \cref{proposition:X_positive}}\label{subsection:proof_X_PD}

First,
suppose that $X_K\succ0$ holds for $K\in\mathbb{R}^{m\times n}$.
Notice that
$(A+BK,W^{1/2})$ is guaranteed to be controllable.
Then, for an eigenvalue $\lambda\in\mathbb{C}$ 
of $(A+BK)^\tr $
with the corresponding unit eigenvector $v$,
pre- and post-multiplying \cref{eq:LQR_unif_Lyapunov} by $v^\her$ and $v$ gives
\begin{equation*}
\begin{cases}
     2\operatorname{Re}(\lambda)v^\her X_K v + v^\her W v=0 & \text{(CT)}\\
     (|\lambda|^2-1 )v^\her X_K v +
     v^\her W v=0. & \text{(DT)}
\end{cases}
\end{equation*}
Now, if $W^{1/2}v=0$, we observe 
$[W^{1/2},(A+BK)W^{1/2},\ldots,(A+BK)^{n-1}W^{1/2}]v=0,$
which contradicts the controllability of $(A+BK,W^{1/2})$.
Hence, 
$W^{1/2}v\neq0$, and thus we have
$v^\her Wv >0$.
Recalling $X_K\succ0$, we get
\begin{equation*}
\begin{cases}
     \operatorname{Re}
     (\lambda)
     = \displaystyle- \frac{v^\her W v}{2v^\her X_K v}<0 & \text{(CT)}\\
     |\lambda|^2 =1\displaystyle- \frac{v^\her W v}{v^\her X_K v}<1 . & \text{(DT)}
\end{cases}
\end{equation*}
Therefore, the matrix $A+BK$ is stable, implying that $K\in\mathcal{K}$.

Next, assume $K\in\mathcal{K}$. 
Then, we have $X_K\succ0$, since
for any nonzero vector $u\in\mathbb{C}^n\setminus\{0\}$, we have
\begin{equation*}
   u^\her  X_K u= 
    \begin{cases}
        \displaystyle \int_0^\infty \|W^{1/2}e^{(A+BK)^\tr t}u\|^2 dt >0& \text{(CT)}\\
         \displaystyle \sum_{t=0}^\infty \left\|W^{1/2}\left((A+BK)^\tr\right)^t u\right\|^2 >0, & \text{(DT)}
    \end{cases}
\end{equation*}
where the functions $\phi_\ct(t):=W^{1/2}e^{(A+BK)^\tr t}u$ and $\phi_\dt(t):=W^{1/2}\left((A+BK)^\tr\right)^t u$
are not uniformly zero by
the controllability of
$(A+BK,W^{1/2})$.
Indeed, for the continuous-time case,
we can compute
\begin{equation*}
 \phi_\ct(0)=W^{1/2}u,\quad
        \phi_\ct^{'}(0) = W^{1/2}(A+BK)^\tr u,\quad\cdots\quad
        \phi^{(n-1)}_\ct(0)
        =W^{1/2}((A+BK)^\tr)^{n-1}u,
\end{equation*}
at least one of which is nonzero.
In the discrete-time case, we have
\begin{equation*}
u^\her X_Ku
\geq \sum_{t=0}^{n-1}
\|\phi_\dt(t)\|^2
= u^\her UU^\tr u^\her >0,
\end{equation*}
where
$U=[W^{1/2},(A+BK)W^{1/2},\ldots,(A+BK)^{n-1}W^{1/2}]$
satisfies $UU^\tr \succ0$.
Hence, we have $u^\her X_Ku>0$ for any $u\neq 0$, which yields $X_K\succ0.$
Hence, we arrive at $K\in\mathcal{K}\Leftrightarrow X_K\succ0$.

\subsection{Proof of \cref{proposition:uniqueness}}\label{subsection:proof_uniqueness}

The first statement is clear from \cref{theorem:LQR_solution_continuous-time,theorem:LQR_solution_discrete-time}
and \cref{example:PL_fail_ct,example:PL_fail_dt}. Since $\mathcal{K}$ is open and $J$ is differentiable, any global minimizer must satisfy $\nabla J(K)=0$.  
We shall prove the uniqueness of the optimal gain by showing that the stationary condition $\nabla J(K)=0$ has a unique solution under \Cref{assumption:X_positive-definite,assumption:stabilizable}.

It is known that
the gradient of the LQR cost at $K\in\mathcal{K}$ is given as follows
\cite{fazel2018global,fatkhullin2021optimizing}:
\begin{align}\label{eq:JK_gradient}
\nabla J(K) =\begin{cases}
     2 (RK+B^\tr P_K)X_K &\text{(CT)}\\
    2\left( (R+B^\tr P_KB)K+B^\tr 
    P_KA\right)X_K,& \text{(DT)}\\
\end{cases}
\end{align}
where $P_K$ solves the following (dual) Lyapunov equation:
\begin{equation}\label{eq:dual_Lyapunov}
    \begin{cases}
        (A+BK)^\tr P_K+P_K(A+BK)+Q+K^\tr RK=0 & 
        \text{(CT)}\\
        (A+BK)^\tr P_K(A+BK)-P+Q+K^\tr RK=0. & \text{(DT)}
    \end{cases}
\end{equation}

We begin with the continuous-time case.
When \cref{assumption:X_positive-definite} holds, we have 
\begin{align*}
    \nabla J(K)=0 \quad \Longleftrightarrow\quad&
     RK+B^\tr P_K = 0\quad
     \Longleftrightarrow\quad
        K=-R^{-1}B^\tr P_K. 
\end{align*}
Then, it is not difficult to see that
 $       (BK)^\tr P_K+P_KBK+K^\tr R K= - P_KBR^{-1}B^\tr P_K.$  
Now, substituting this into
\cref{eq:dual_Lyapunov} yields
\begin{align*}
        A^\tr P_K+P_KA+Q- P_KBR^{-1}B^\tr P_K=0.
\end{align*}
Therefore, by inspecting the ARE in \cref{theorem:LQR_solution_continuous-time}
and by recalling the uniqueness of the solution $P_K$ to \cref{eq:dual_Lyapunov} and that of $P^\star$ to the ARE, 
we observe $P_K = P^\star$, and
        $K=-R^{-1}B^\tr P^\star$
is indeed the unique solution to the LQR problem \cref{eq:LQR_unif}
for the continuous-time case.

Finally, the discrete-time case follows from the same argument.
Under \cref{assumption:X_positive-definite},
\begin{align*}
    \nabla J(K)=0 \quad \Longleftrightarrow\quad&
        (R+B^\tr P_KB)K+B^\tr P_KA=0
     \quad\Longleftrightarrow\quad 
        K=-(R+B^\tr P_KB)^{-1}B^\tr P_KA,
\end{align*}
which yields
\begin{align*}
        (BK)^\tr P_KA+A^\tr P_K BK +K^\tr B^\tr P_KBK+ K^\tr RK
        = 
        - A^\tr P_KB(R+B^\tr P_KB)^{-1}B^\tr P_K A.
\end{align*}
Now, applying this to
\cref{eq:dual_Lyapunov}, we have
\begin{align*}
        A^\tr P_K A- P_K+Q- A^\tr P_KB(R+B^\tr P_KB)^{-1}B^\tr P_KA=0.
\end{align*}
Hence, recalling the DARE in \cref{theorem:LQR_solution_discrete-time}
and the uniqueness of the solution $P_K$ to \cref{eq:dual_Lyapunov}, 
we confirm $P_K = P^\star$ and
$K=-(R+B^\tr P^\star B)^{-1}B^\tr P^\star A$,
which is the unique stationary and optimal LQR~gain.

%% file: sections/appendix-proofs.tex
\section{Technical proofs in \cref{section:proof_PL,section:proof_partialmin_chain-rule}}\label{appendix:proof}

\subsection{Proof of \Cref{proposition:LQR_KX_ineq}} \label{appendix:proof-LQR-KX-ineq}

By \Cref{assumption:X_positive-definite}, we can directly replace the constraint $K \in \mathcal{K}$ with $X \succ 0$ in \cref{eq:LQR_unif_Lyapunov} so that the optimal solution and optimal value remain the same. 

It suffices to prove that relaxing the Lyapunov equation in \cref{eq:LQR_unif_Lyapunov} into the Lyapunov inequality in \cref{eq:LQR_unif_2-lyapunov} does not affect the optimal gain $K^\star$ and the optimal value. This indeed follows from a standard comparison result in Lyapunov equations \cite{lancaster1995algebraic,sun2021learning}. 
As shown in \cref{proposition:Lyapunov_general},
the following fact holds for both continuous-time and discrete-time settings: 
fix $K\in\mathcal{K}$, and let $W_1,W_2$ satisfying $W_1\succeq W_2\succeq0$, then 
the unique solutions $X_i\,(i=1,2)$ to the Lyapunov equations
\begin{align*}
     \Psi
   \left(
   \begin{bmatrix}
       A+BK\\
       I
   \end{bmatrix}X_i
   \begin{bmatrix}
       A+BK\\
       I
   \end{bmatrix}^\tr
   \right) +W_i=0,\quad i=1,2
\end{align*}
satisfy $X_1\succeq X_2\succeq 0$.
By this property, it is not difficult to verify that for any feasible $(K, X)$ in \cref{eq:LQR_unif_2-lyapunov}, we have
$
    X\succeq X_K, 
$ 
where $X_K$ is the unique Lyapunov solution to \cref{eq:LQR_unif_Lyapunov} for the same $K$.
Thus, we have 
\begin{equation} \label{eq:LQR-two-upper-bound}
     \langle Q+K^\tr RK,X\rangle\geq \langle Q+K^\tr RK,X_K\rangle
\end{equation}
for any feasible  $(K, X)$ in \cref{eq:LQR_unif_2-lyapunov}. 

Now, suppose $(K^\star,X^\star)$ is an optimal solution to \cref{eq:LQR_unif}. This pair $(K^\star,X^\star)$ is also feasible to \cref{eq:LQR_unif_2}, and it must also be an optimal solution to \cref{eq:LQR_unif_2} by the inequality \Cref{eq:LQR-two-upper-bound}. Problems \cref{eq:LQR_unif_2,eq:LQR_unif} have the same optimal value in this case. 
On the other hand, suppose $(K^\star,X^\star)$ is an optimal solution to \cref{eq:LQR_unif_2}. This gain $K^\star$ must be stabilizing, and optimal to \cref{eq:LQR_unif} (otherwise, there exists another gain $K$ that achieve a lower value in \cref{eq:LQR_unif}, and this gain $K$ together with $X_K$ would achieve a lower value in \cref{eq:LQR_unif_2}, contradicting the optimality of $(K^\star,X^\star)$). 
We next show that the optimal value between \cref{eq:LQR_unif_2,eq:LQR_unif} is the same. We construct another unique $X^\star_K \succ 0$ from the Lyapunov equation \cref{eq:LQR_unif_Lyapunov} by \Cref{assumption:X_positive-definite}. This pair $(K^\star,X^\star_K)$ is also feasible to \cref{eq:LQR_unif_2}, and we thus have 
$$
\langle Q+(K^\star)^\tr RK^\star,X^\star\rangle\leq \langle Q+(K^\star)^\tr RK^\star,X^\star_K\rangle.
$$
Together with \Cref{eq:LQR-two-upper-bound}, we must have $\langle Q+(K^\star)^\tr RK^\star,X^\star\rangle = \langle Q+(K^\star)^\tr RK^\star,X^\star_K\rangle$.

\subsection{Convexity of 
the problem \cref{eq:LQR_unif_convex} in discrete-time}\label{subsec-proof:convex_derivation}

We first derive another convex reformulation 
for \cref{eq:LQR_unif_2} in discrete-time
with an additional variable $Z$, and derive \cref{eq:LQR_unif_convex} by partially minimizing it with respect to $Z$.
Since partial minimization preserves convexity (see \cite[Prop. 2.22]{rockafellar1998variational}),
we can conclude that \cref{eq:LQR_unif_convex} is indeed a convex program.

For 
$(K,X)$ feasible to
\cref{eq:LQR_unif_2},
we introduce the following variables $Y$ and $Z$:
\begin{align}\label{eq:convex_Z}
    Y=&(K-K^\star)X, \nonumber \\
    Z\succeq &(Y+K^\star X) X^{-1}(Y+K^\star X)^\tr
    \quad(=KXK^\tr)
    .
\end{align}
Then, we can reformulate \cref{eq:LQR_unif_2} with $\Psi=\Psi_{\dt}$ as
\begin{equation}\label{eq:convex_YXZ-1}
    \begin{aligned}
        \min_{Y,\,X,\,Z}
    &\quad 
    \langle Q ,X\rangle
    +
    \langle R ,Z\rangle
    \\
    \text{subject to}&\quad
    X\succ 0,
    \quad \cref{eq:convex_Z},
    \quad \Lambda(Y,X)\preceq 0
    \end{aligned}
\end{equation}
with
    \begin{align*}
    \Lambda(Y,X)
    =& \Psi_{\dt}
    \left(
    \begin{bmatrix}
        AY+BK^\star\\
        I
    \end{bmatrix}
    X^{-1}
    \begin{bmatrix}
        AY+BK^\star \\
        I
    \end{bmatrix}^\tr
    \right)
    +W\\
    =&
    (A^\star X+BY)X^{-1}(A^\star X+BY)^\tr-
    X +W,
\end{align*}
where $A^\star = A+BK^\star$.
Note that when $ Z = 
    (Y+K^\star X) X^{-1}(Y+K^\star X)^\tr,$
this problem is identical to \cref{eq:LQR_unif_convex}.
Then, since the Schur complement gives the following LMIs
\begin{align*}
\Lambda(Y,X)\preceq 0
\quad\Longleftrightarrow&\quad
\begin{bmatrix}
    X-W & (A^\star X+BY)\\
    (A^\star X+BY)^\tr & X
\end{bmatrix}\succeq 0\\
\cref{eq:convex_Z}\quad \Longleftrightarrow&\quad \begin{bmatrix}
    Z & (Y+K^\star X)\\
    (Y+K^\star X)^\tr &X
\end{bmatrix}\succeq 0,
\end{align*}
and $Z = 
    (Y+K^\star X) X^{-1}(Y+K^\star X)^\tr$
constitutes a partial minimizer for \cref{eq:convex_YXZ-1} w.r.t. $Z$,
we conclude that 
both \cref{eq:LQR_unif_convex-cost} and \cref{eq:LQR_unif_convex_Lyapunov} are convex.

\subsection{Computational details of the chain rule }\label{subsec-proof:chain-rule}

\subsubsection{Derivation of \cref{eq:chain-rule_vec} via vectorization}\label{subsubsec-proof:chain-rule}
Recall that
$\Pi(Y,X) = 
    (YX^{-1}+K^\star,X)$.
As stated in the proof of \cref{proposition:fcvx_subgradient_J}, it holds 
for any feasible $(Y,X)\in\mathcal{F}_\mathrm{cvx}$ that
\begin{equation}\label{eq:fcvx_composite}
    \tilde{f}_\mathrm{cvx}(Y,X)
    =
    (J_\mathrm{lft}\circ\Pi)(Y,X).
\end{equation}
Now, 
to use the chain rule,
define
 the following vector-valued functions and mapping:
\begin{align*}
\hat{f}_\mathrm{cvx}\left(\begin{bmatrix}
         y \\
         x
\end{bmatrix}\right)=&\tilde{f}_\mathrm{cvx}\left(
    \operatorname{mat}(y),
    \operatorname{mat}(x)\right),\,\\
    \hat{J}_\mathrm{lft}\left(\begin{bmatrix}
        k\\
        x
    \end{bmatrix}\right) = &
    J_\mathrm{lft} 
    \left(
    \operatorname{mat}(k),
    \operatorname{mat}(x)
\right),\quad
\hat{\Pi}\left(\begin{bmatrix}
         y \\
         x
\end{bmatrix}\right)=
\begin{bmatrix}
   \operatorname{vec}\left(\operatorname{mat}(y)\operatorname{mat}(x)^{-1}+K^\star\right)\\
    x
\end{bmatrix},
\end{align*}
where $\operatorname{vec}^{-1}(\cdot)=\operatorname{mat}(\cdot).$
Then, we obtain
\begin{equation*}
     \hat{f}_\mathrm{cvx}
     \left(\begin{bmatrix}
         \operatorname{vec}(Y) \\
         \operatorname{vec}(X)
     \end{bmatrix}\right)=
     (
    \hat{J}_\mathrm{lft}
    \circ
    \hat{\Pi}
    )
    \left(
    \begin{bmatrix}
        \operatorname{vec}(Y)\\
        \operatorname{vec}(X)
    \end{bmatrix}
    \right)
\end{equation*}
by
\begin{align*}
    \tilde{f}_\mathrm{cvx}(Y,X)
    =
     \hat{f}_\mathrm{cvx}
     \left(\begin{bmatrix}
         \operatorname{vec}(Y) \\
         \operatorname{vec}(X)
     \end{bmatrix}\right)
\overset{\cref{eq:fcvx_composite}}{=}
     &(J_\mathrm{lft} \circ
    \Pi)(Y,X)
    =
    (
    \hat{J}_\mathrm{lft}
    \circ
    \hat{\Pi}
    )
    \left(
    \begin{bmatrix}
        \operatorname{vec}(Y)\\
        \operatorname{vec}(X)
    \end{bmatrix}
    \right).
\end{align*}

We compute \cref{eq:chain-rule_vec} using this vectorized form.
Note that
$\hat{\Pi}$
are continuously differentiable on $\operatorname{vec}(\mathbb{R}^{m\times n})\times \operatorname{vec}(\mathbb{S}_{++}^n)$,
and
$J_\mathrm{lft}$ is lower semi-continuous at $\Pi
        (Y,X)$ for any  $(Y,X) \in \mathcal{F}_\mathrm{cvx}$.
By the chain rule \cite[Th. 10.6]{rockafellar1998variational} 
for $\hat{f}_\mathrm{cvx}=\hat{J}_\mathrm{lft}\circ\hat{\Pi}$,
we have
\begin{equation}\label{eq:chain-rule_proof-1}
\Gamma(Y,X)^\tr 
\hat{\partial}
\hat{J}_\mathrm{lft}
\left( \hat{\Pi}\left(
    \begin{bmatrix}
        \operatorname{vec}(Y)\\
        \operatorname{vec}(X)
    \end{bmatrix}
    \right) \right)
    \subset \partial \hat{f}_\mathrm{cvx}
    \left(
    \begin{bmatrix}
       \operatorname{vec}(Y)\\
        \operatorname{vec}(X)
    \end{bmatrix}
    \right),\qquad
    \Gamma(Y,X)
    =
    \nabla \hat{\Pi}\left(
    \begin{bmatrix}
        \operatorname{vec}(Y)\\
        \operatorname{vec}(X)
    \end{bmatrix}
    \right)
\end{equation}
for $(Y,X)\in\mathcal{F}_\mathrm{cvx}$.
We present the explicit form of $\Gamma(Y,X)$ in \cref{proposition:Jacobian}.
Since $\operatorname{vec}(\cdot)$ only rearranges matrices to be vectors,
it follows 
that
\begin{equation}\label{eq:chain-rule_proof-2}
\partial \tilde{f}_\mathrm{cvx}\left(
Y,X
    \right) = 
    \left\{
    (H_1,
        H_2)
    \middle|
        \begin{bmatrix}
        \operatorname{vec}(H_1)\\
        \operatorname{vec}(H_2)
    \end{bmatrix}
        \in
    \partial \hat{f}_\mathrm{cvx}\left(
    \begin{bmatrix}
        \operatorname{vec}(Y)\\
        \operatorname{vec}(X)
    \end{bmatrix}
    \right)
    \right\}.
\end{equation}
Similar to this, we have
\begin{equation}\label{eq:chain-rule_proof-3}
    \hat{\partial}
    {J}_\mathrm{lft}
\left( \Pi(Y,X)
    \right)
    = 
    \left\{
    (\Xi_1,\Xi_2)
    \middle|
        \begin{bmatrix}
        \operatorname{vec}(\Xi_1)\\
        \operatorname{vec}(\Xi_2)
    \end{bmatrix}
        \in
     \hat{\partial} \hat{J}_\mathrm{lft}\left(
     \hat{\Pi}\left(
    \begin{bmatrix}
        \operatorname{vec}(Y)\\
        \operatorname{vec}(X)
    \end{bmatrix}
    \right) 
     \right)
    \right\}.
\end{equation}
Combining \cref{eq:chain-rule_proof-1}, \cref{eq:chain-rule_proof-2}, and \cref{eq:chain-rule_proof-3},
we can rewrite \cref{eq:chain-rule_proof-1} into \cref{eq:chain-rule_vec}.

\subsubsection{Proof of \cref{proposition:Jacobian}}\label{subsubsec-proof-Jacobian}

Here, we compute the following Jacobian explicitly:
\begin{equation*}
    \nabla \hat{\Pi}\left(
    \begin{bmatrix}
        \operatorname{vec}(Y)\\\operatorname{vec}(X)
    \end{bmatrix}
    \right)
    = \begin{bmatrix}
    X^{-1}\otimes I_m
    &
    -(X^{-1}\otimes YX^{-1})\\
    0& I_{n^2}
    \end{bmatrix}. 
\end{equation*}
First, it is easy to compute the lower two blocks as
\begin{align*}
      \frac{\partial }{\partial \operatorname{vec}^\tr(Y) }
          \operatorname{vec}(X)=0,\qquad
         \frac{\partial }{\partial \operatorname{vec}^\tr(X) }
         \operatorname{vec}(X)
         = I_{n^2}.
\end{align*}
Next,
for the upper left block, using the formula $(S_2^\tr\otimes S_1)\operatorname{vec}(V)
=\operatorname{vec}(S_1VS_2)$ \cite[Prop. 5.2.1]{lancaster1995algebraic} gives
\begin{align*}
    \operatorname{vec}(YX^{-1}+K^\star)
    =&
    \operatorname{vec}(YX^{-1}) + \text{const.}\\
    =&( X^{-1}\otimes I_m)\operatorname{vec}(Y)+\text{const.}
    \\
    =&(I_n\otimes Y)\operatorname{vec}(X^{-1})+\text{const.},
\end{align*}
which yields
\begin{align*}
    \frac{\partial }{\partial \operatorname{vec}^\tr(Y) }
         \operatorname{vec}
         (YX^{-1}+K^\star)
         =& X^{-1}\otimes I_m.
\end{align*}
Finally, we can compute the upper right block as
\begin{align*}
         \frac{\partial }{\partial \operatorname{vec}^\tr(X) }
         \operatorname{vec}
         (YX^{-1}+K^\star)=& 
         - (I_n\otimes Y)( X^{-1}\otimes X^{-1})
         =-(X^{-1}\otimes YX^{-1}),
\end{align*}
where we have used
\begin{equation}\label{eq:Xinv_partial}
    \frac{\partial \operatorname{vec}(X^{-1}) }{\partial \operatorname{vec}^\tr (X) }
    =-( X^{-1}\otimes X^{-1}).
\end{equation}
Note that \cref{eq:Xinv_partial} follows from
\begin{align*}
    &(X+dX)^{-1}
    =
    X^{-1}(I+dX\,X^{-1})^{-1}
    =X^{-1} -X^{-1}\,dX\,X^{-1} + O(\|dX\|_F^2) \\
    \Longrightarrow\quad
&    \operatorname{vec}\left((X+dX)^{-1}\right)
    =\operatorname{vec}(X^{-1})
    - (X^{-1}\otimes X^{-1})\operatorname{vec}(dX)
    + O(\|dX\|_F^2).
\end{align*}

\subsection{Algebraic proof without 
chain rules for \cref{proposition:fcvx_subgradient_J}}\label{subsection:proof-2-fcvx_subgradient}

We provide an alternative proof for \Cref{proposition:fcvx_subgradient_J}
without using the chain rule.
We use some of the notations introduced in \cref{section:proof_partialmin_chain-rule}.

Recall 
that $J_\mathrm{lft}(K,X):= \langle Q+K^\tr RK,X \rangle + \delta_{\mathcal{L}_\mathrm{lft}}(K,X)$
with the set $\mathcal{L}_\mathrm{lft}$ of all feasible points of \cref{eq:LQR_unif_2},
and we have
\begin{equation*}
    J_\mathrm{lft}(K,X)
    =
    (\tilde{f}_\mathrm{cvx}\circ\Upsilon)
    (K,X)
\end{equation*}
for any $(K,X)\in\mathbb{R}^{m\times n}\times \mathbb{S}_{++}^n$.
Moreover, recalling 
Fre\'chet subgradients
(\Cref{definition:Frechet_subgradient}) and
\cref{proposition:subgradient_Jlft},
we know $(\nabla J(K),0)\in\hat{\partial}J_\mathrm{lft}(K,X_K)$, which follows from
\cref{eq:Frechet_Jlft}:
\begin{equation}
  J_\mathrm{lft}(\hat{K},\hat X)
        \geq
        J_\mathrm{lft}(K,X_K)
        + \langle 
        \nabla J(K),\hat K-K
        \rangle
        +
        \langle 
        0,\hat X-X_K
        \rangle
        +o(\|\hat K-K\|_F)
    \tag{\cref{eq:Frechet_Jlft} restated}
\end{equation}
for any $(\hat{K},\hat{X})$ near $(K,X_K)$.

    Using this, we establish \cref{proposition:fcvx_subgradient_J} as follows.
    Let $(\hat Y,\hat{X})\in \mathbb{R}^{m\times n}\times \mathbb{S}_{++}^n$.
    For $(K,\hat{X})=\Upsilon^{-1}(\hat{Y},\hat{X})$,
    we have
    \begin{equation*}
        \hat{Y}=(\hat {K}-K^\star)\hat{X},\qquad
        \hat{K}=\hat{Y}\hat{X}^{-1}+K^\star.
    \end{equation*}
    Then, with $Y_K=(K-K^\star )X_K$,
    it is not difficult to see that
    \begin{equation}\label{eq:JK_inprod_proof}
    \begin{aligned}
        &\langle 
        \nabla J(K),\hat K-K
        \rangle\\
        =&
         \Tr 
        \left(
        (\nabla J(K)X_K^{-1})^\tr (\hat K-K)X_K
        \right) \\
        =& 
        \Tr \left(
        (\nabla J(K)X_K^{-1})^\tr  \left((\hat K-K^\star)-(K-K^\star)\right)(\hat X+X_K-\hat{X})
        \right) \\
        \overset{(a)}{=}& 
    \left\langle 
        \nabla J(K)X_K^{-1},
        \hat{Y}-Y_K\right\rangle
        +
        \Tr\left(
        (\nabla J(K)X_K^{-1})^\tr
        (\hat{K}-K^\star) (X_K-\hat{X})
        \right)\\
        =& \left\langle 
        \nabla J(K)X_K^{-1},
        \hat{Y}-Y_K
        \right\rangle+
        \Tr\left( 
        (\nabla J(K)X_K^{-1})^\tr
        \left(K-K^\star+\hat{K}-K\right) (X_K-\hat{X})
        \right)\\
        =& \left\langle 
        \nabla J(K)X_K^{-1},
        \hat{Y}-Y_K\right\rangle+
        \left\langle 
         (K^\star-K)^\tr \nabla J(K)X_K^{-1},\hat{X}-X_K
        \right\rangle
        +O\left(\|\hat K-K\|_F\|\hat X-X_K\|_F\right),
    \end{aligned}
    \end{equation}
    where
    (a) follows from $
        \hat{Y}-Y_K
        =(\hat K-K^\star)\hat{X}-(K-K^\star)X_K$.
    Now, $O(\|\hat K-K\|_F)$ is of the same order as
    $O\left(
        \left\|\begin{bmatrix}
            \hat  Y- Y_K\\
            \hat  X- X_K
        \end{bmatrix}\right\|_F
        \right)$,
        as
    \begin{align*}
        \hat K-K
        =&\hat Y\hat X^{-1}-Y_KX_K^{-1}\\
        =& (\hat Y-Y_K)X_K^{-1}
        +\hat{Y}X_K^{-1}(X_K-\hat{X})X_K^{-1}
        + O(\|\hat X-X_K\|_F^2),
    \end{align*}
    where we have used
    the following transformation by the Neumann series:
    \begin{align*}
        \hat X^{-1}
        = X_K^{-1}\left(
        I-(X_K-\hat{X})X_K^{-1}
        \right)^{-1} 
        =X_K^{-1}
        + X_K^{-1}(X_K-\hat{X})X_K^{-1}+O(\|\hat X-X_K\|_F^2).
    \end{align*}
    Thus, plugging
    in
    \cref{eq:JK_inprod_proof},
$\tilde{f}_\mathrm{cvx}(\hat{Y},\hat{X})=J_\mathrm{lft}(\hat{K},\hat{X})$,
    and $J(K)=\tilde{f}_\mathrm{cvx}(Y_K,X_K)$ 
    for 
\cref{eq:Frechet_Jlft}
    yields
    \begin{align*}
        &\tilde{f}_\mathrm{cvx}(\hat Y,\hat{X})
        \geq
        \tilde{f}_\mathrm{cvx}(Y_K,X_K)\\
        &
       + \left\langle 
        \nabla J(K)X_K^{-1},
        \hat{Y}-Y_K\right\rangle+
        \left\langle 
         (K^\star-K)^\tr \nabla J(K)X_K^{-1},\hat{X}-X_K
        \right\rangle
        +
        o\left(
        \left\|\begin{bmatrix}
            \hat  Y- Y_K\\
            \hat  X- X_K
        \end{bmatrix}\right\|_F
        \right)
    \end{align*}
    for any $(\hat Y,\hat X)$ near $(Y_K,X_K)$.
    Notice that this inequality still holds 
    even for $\hat X\in \mathbb{R}^{n\times n}\setminus\mathbb{S}_{++}^n$,
    since we know
    $\tilde{f}_\mathrm{cvx}(\hat Y,\hat{X})=\infty$ for this $\hat X$.
    Therefore, \Cref{definition:Frechet_subgradient}
    yields \cref{eq:gradJ_fcvx}, i.e.,
    \begin{equation*}
        (\nabla J(K)X_K^{-1},
         (K^\star-K)^\tr \nabla J(K)X_K^{-1})
         \in \partial \tilde{f}_\mathrm{cvx}(Y_K,X_K),
    \end{equation*}
    which implies \cref{proposition:fcvx_subgradient_J}.